\documentclass[reqno]{amsart}
\usepackage{amsmath,amsfonts,amssymb,amsthm}
\usepackage[usenames,dvipsnames]{xcolor}
\usepackage[colorlinks=true]{hyperref} 
\usepackage{verbatim}
\usepackage{graphicx}
\hypersetup{linkcolor=Violet,citecolor=PineGreen}
\newtheorem{thm}{Theorem}[section]

\newtheorem{prop}[thm]{Proposition}
\newtheorem{cor}[thm]{Corollary}
\theoremstyle{definition}
\newtheorem{defn}[thm]{Definition}
\newtheorem{exmpl}[thm]{Example}
\newtheorem{rmk}[thm]{Remark}
\numberwithin{equation}{section}
\newcommand{\Hu}{\mathbb{H}}
\newcommand{\N}{\mathbb{N}}
\newcommand{\R}{\mathbb{R}}
\newcommand{\Q}{\mathbb{Q}}
\newcommand{\Z}{\mathbb{Z}}

\newcommand{\K}{\mathcal{K}}
\newcommand{\M}{\mathcal{M}}
\newcommand{\T}{\mathcal{T}}
\newcommand{\U}{\mathcal{U}}
\newcommand{\ep}{\varepsilon}
\newcommand{\normp}[2]{\left\| #1 \right\|_{p}} 

\newcommand{\C}{\mathfrak{C}_p}
\newcommand{\LC}{\mathfrak{LC}_p}
\newcommand{\SC}{\mathfrak{SC}_p}

\newcommand{\nbd}{\nobreakdash}
\newcommand{\TC}{\Theta\C}

\newcommand{\nin}{{n\in\N}}
\newcommand{\nti}{{n\to\infty}}

\DeclareMathOperator{\cls}{cls}
\newcommand{\meas}{\mathrm{meas}}
\DeclareMathOperator{\supp}{supp}
\begin{document}
\title[Topologies of $L^p_{loc}$ type for Carath\'eodory functions in ODE$\text{s}$]
{Topologies of $L^p_{loc}$ type for Carath\'eodory functions with applications in non-autonomous differential equations}
\author[I.P.~Longo]{Iacopo P. Longo}
\address{Departamento de Matem\'{a}tica Aplicada, Universidad de
Valladolid, Paseo del Cauce 59, 47011 Valladolid, Spain.}
\email[Iacopo P. Longo]{iaclon@wmatem.eis.uva.es}
\email[Sylvia Novo]{sylnov@wmatem.eis.uva.es}
\email[Rafael Obaya]{rafoba@wmatem.eis.uva.es}
\thanks{Partly supported by MINECO/FEDER
under project MTM2015-66330-P and EU Marie-Sk\l odowska-Curie ITN Critical Transitions in
Complex Systems (H2020-MSCA-ITN-2014 643073 CRITICS)}
\author[S.~Novo]{Sylvia Novo}
\author[R.~Obaya]{Rafael Obaya}
\date{}
\begin{abstract}
Metric topological vector spaces of Carath\'eodory functions and topologies of $L^p_{loc}$ type are introduced, depending on a suitable set of moduli of continuity. Theorems of continuous dependence on initial data for the solutions of non-autonomous Carath\'eodory differential equations are proved in such new topological structures. As a consequence, new families of continuous linearized skew-product semiflows are provided in the Carath\'eodory spaces.
\end{abstract}
\maketitle
\section{Introduction}\label{secintro}
In this paper we introduce new topologies of $L^p_{loc}$ type in order to study the behavior of the solutions of non-autonomous Carath\'eodory differential equations. The problem is classic and it was firstly introduced by Miller and Sell~\cite{book:RMGS,paper:RMGS1}. Since then, $L^p_{loc}$ topologies have been employed to investigate non-autonomous linear differential equations (see Bodin and Sacker~\cite{paper:SBRS}, Chow and Leiva~\cite{paper:CHLE} and Siegmund~\cite{paper:SS} among others) but, despite its potential interest, the classic theory has not been conveniently developed in the field of non-linear differential equations.
\par
\vspace{0.05cm}
To this aim, we introduce new dynamical arguments, filling some gaps in the theory and improving its applicability. In particular, we define new topologies and new locally convex vector spaces where the flow map defined by the time-translation proves to be continuous, and deduce theorems of continuous dependence with respect to the variation of initial data for the solutions of differential problems whose vector fields belong to such spaces. The continuity of the skew-product flow composed by the base flow on the \emph{hull} of a vector field, and by the solutions of the respective differential problem, is also achieved. As a consequence of the previous results, a range of dynamical scenarios is opened in which it is possible to combine techniques of continuous skew-product flows, processes and random dynamical systems (see Arnold~\cite{book:LA}, Aulbach and  Wanner~\cite{paper:AW}, Johnson et al.~\cite{book:JONNF},  Carvalho et al.~\cite{book:CLR}, P\"otzsche and Rasmussen~\cite{paper:CPMR}, Sell~\cite{book:GS},  Shen and Yi~\cite{paper:WSYY} and the references therein).
\par
\vspace{0.05cm}
The structure and the main results of the paper are organized as follows. Section~\ref{sectopo}
is devoted to recall the topologies $\T_B$ and $\T_D$ on the space $\SC$ of strong Carath\'eodory functions which were firstly presented in the classic references~\cite{book:RMGS,paper:RMGS1}, as well as to introduce the new spaces of $\Theta$-Carath\'eodory functions and the respective new topologies denoted by $\T_\Theta$, where $\Theta$ is a suitable set of moduli of continuity determined by the non-linear equations to be studied. The locally convex metric spaces $(\TC,\T_\Theta)$ will be essential in our paper.
Additionally, the symbol $\LC$ will denote the space of Lipschitz Carath\'eodory functions where the restriction of the previously outlined topologies will play an important role.
\par\vspace{0.05cm}
Section~\ref{sectrans} mainly deals with proving that the map defined by the time-translation provides a continuous flow on $\TC$. Such a result paves the way to the introduction of the concept of the hull of a function in $(\TC,\T_\Theta)$. The definitions of the hulls in $\SC$ and $\LC$ with respect to any of the considered topologies are also recalled.\par\vspace{0.05cm}
In Section~\ref{secTopml} we investigate the topological properties of Carath\'eodory functions admitting $L^p_{loc}$-bounded (resp. $L^1_{loc}$-equicontinuous) $m$-bounds and/or $l$-bounds. In particular, we prove that if $E\subset \SC$ has $L^p_{loc}$-bounded (resp. $L^1_{loc}$-equicontinuous) $m$-bounds and/or $l$-bounds, then  such a property is also inherited by the closure of $E$ in $(\SC,\T)$, where $\T$ is any of the topologies introduced in Section~\ref{sectopo}. As a corollary, we have the corresponding versions of these results for $E=\mathrm{Hull}_{(\SC,\T)}(f)$ where $f\in\SC$ has adequate $m$-bounds and/or $l$-bounds. Furthermore, we prove that, if $E\subset\LC$ has $L^p_{loc}$-bounded  $l$-bounds, then $\mathrm{cls}_{(\SC,\T_1)}(E)=\mathrm{cls}_{(\SC,\T_2)}(E)$, where $\T_1$ and $\T_2$ are any of the previously introduced topologies. We conclude the section giving a sufficient condition for the relative compactness of a set  $E\subset\LC$ with respect to any of the previous topologies under the assumption that $E$ admits $L^p_{loc}$-bounded  $l$-bounds. The problems considered in this section were initially posed by Artstein~\cite{paper:ZA1,paper:ZA2} and Sell~\cite{paper:GS1}.
\par \vspace{0.05cm}
Section~\ref{secContFlowODEs} is devoted to the study of  triangular Carath\'eodory systems of the type $\dot x=f(t,x)$, $\dot y=F(t,x)y+h(t,x)$, where $f\in\LC$   admits either $L^1_{loc}$-equicontinuous $m$-bounds, or $L^1_{loc}$-bounded $l$-bounds. We determine a suitable set of moduli of continuity $\Theta$, starting, in the first case, from the $L^1_{loc}$-equicontinuous $m$-bounds and, in the second one, from the solutions of the differential equations $\dot x=f(t,x)$ when $f$ is in a compact subset of $\LC$. The functions $F$ and $h$ are chosen to be in  $\TC\big(\R^{N\times N}\big)$ and $\TC$, respectively. In each of the two considered scenarios, we give sufficient conditions for the continuity of the solutions with respect to the variation of the initial conditions firstly, and then for the continuity of the skew-product semiflow composed by the base flow on the $\mathrm{Hull}_{(\LC\times \TC\times\TC)}(f,F,h) $ and the solutions of the respective differential equations.
\par\vspace{0.05cm}
In Section~\ref{sec:linearizedSPflow}, assuming that $f$ admits $L^1_{loc}$-equicontinuous $m$-bounds and continuous partial derivatives with respect to $x$, that $J_xf\in\SC$, and using the results obtained in Section~\ref{secContFlowODEs}, we prove the existence of the linearized skew-product semiflow composed by the base flow on the $\mathrm{Hull}_{(\LC\times \TC,\T_\Theta)}(f,J_xf) $ and the solutions of the respective differential equations. In particular, we show that the solutions of Carath\'eodory differential equations are differentiable with respect to initial data even in some cases in which the vector field  has not continuous partial derivative with respect to $x$. \par\vspace{0.05cm}
Finally, notice that the continuous variation of the solutions of Carath\'eodory differential equations  has been widely investigated when weak topologies are considered (see Artstein~\cite{paper:ZA1,paper:ZA2, paper:ZA3}, Heunis~\cite{paper:AJH} and Neustadt~\cite{paper:LWN} among others). The use of some of the ideas contained in this paper and the employment of weak and strong $L^p_{loc}$-like topologies in the study of  Carath\'eodory functional differential equations will be the contents of a forthcoming publication.
\section{Spaces and topologies}\label{sectopo}
In the following, we will denote by $\R^N$ the $N$\nbd-dimensional euclidean space with norm $|\cdot|$ and by $B_r$ the closed ball of $\R^N$ centered at the origin and with radius $r$. When $N=1$ we will simply write $\R$ and the symbol $\R^+$ will denote the set of positive real numbers. Moreover, for any interval $I\subseteq\R$ and any $W\subset\R^N$, we will use the following notation
\begin{itemize}
\item[] $C(I,W)$: space of continuous functions from $I$ to $W$ endowed with the norm $\|\cdot\|_\infty$.
\item[] $C_C(\R)$: space of continuous functions with compact support in $\R$, endowed with the norm $\|\cdot\|_\infty$. When we want to restrict to the positive continuous functions with compact support in $\R$, we will write  $C^+_C(\R)$.
\item[] $L^p(I,\R^N)$, $1\le p <\infty$: space of measurable functions from $I$ to $\R^N$ whose norm is in the Lebesgue space $L^p(I)$.
\item[] $L^p_{loc}(\R^N)$, $1\le p <\infty$: the space of all functions $x(\cdot)$ of $\R$ into $\R^N$ such that for every compact interval $I\subset\R$, $x(\cdot)$ belongs to $L^p\big(I,\R^N\big)$. When $N=1$, we will simply write $L^p_{loc}$.
\end{itemize}
Let $1\le p <\infty$; we will consider, and denote by $\C\big(\R^M\big)$ (or simply $\C$ when $M=N$), the set of functions $f\colon\R\times\R^N\to \R^M$ satisfying
\begin{itemize}
\item[(C1)] $f$ is Borel measurable and
\item[(C2)] for every compact set $K\subset\R^N$ there exists a real-valued function $m^K\in L^p_{loc}$, called \emph{$m$-bound} in the following, such that $|f(t,x)|\le m^K(t)$ for any $x\in K$ and almost every $t\in\R$.
\end{itemize}
Now we introduce the sets of Carath\'eodory functions which are subsequently used.
\begin{defn}\label{def:LC}
A function $f\colon\R\times\R^N\to \R^M$ is said to be \emph{Lipschitz Carath\'eodory for $1\le p <\infty$}, and we will write $f\in \LC\big(\R^M\big)$ (or simply $f\in\LC$ when $M=N$), if it satisfies (C1), (C2) and
\begin{itemize}
\item[(L)] for every compact set $K\subset\R^N$ there exists a real-valued function $l^K\in L^p_{loc}$ such that $|f(t,x)-f(t,y)|\le l^K(t)|x-y|$ for any $x,y\in K$ and almost every $t\in\R$.
\end{itemize}
In particular, for any compact set $K\subset\R^N$, we refer to \emph{the optimal $m$-bound} and \emph{the optimal $l$-bound} of $f$ as to
\begin{equation}
m^K(t)=\sup_{x\in K}|f(t,x)|\qquad \mathrm{and}\qquad l^K(t)=\sup_{\substack{x,y\in K\\ x\neq y}}\frac{|f(t,x)-f(t,y)|}{|x-y|}\, ,
\label{eqOptimalMLbound}
\end{equation}
respectively. Clearly, for any compact set $K\subset\R^N$ the suprema in \eqref{eqOptimalMLbound}  can be taken for a countable dense subset of $K$ leading to the same actual definition, which grants that the functions defined in \eqref{eqOptimalMLbound} are measurable.
\end{defn}
\begin{defn}\label{def:SC}
A function $f\colon\R\times\R^N\to \R^M$ is said to be \emph{strong Carath\'eodory for $1\le p <\infty$}, and we will write $f\in \SC\big(\R^M\big)$ (or simply $f\in\SC$ when $M=N$), if it satisfies (C1), (C2) and
\begin{itemize}
\item[(S)] for almost every $t\in\R$, the function $f(t,\cdot)$ is continuous.
\end{itemize}
The concept of \emph{optimal $m$-bound} for a strong  Carath\'eodory function on any compact set $K\subset\R^N$, is defined exactly as in equation \eqref{eqOptimalMLbound}.
\end{defn}
Functions which are not necessarily continuous in the second variable, are also considered. In order to define such a set, we need to set some notation.
\begin{defn}
We call  \emph{a suitable set of moduli of continuity}, any countable  set of non-decreasing continuous functions
\begin{equation*}
\Theta=\left\{\theta^I_j \in C(\R^+, \R^+)\mid j\in\N, \ I=[q_1,q_2], \ q_1,q_2\in\Q\right\}
\end{equation*}
such that $\theta^I_j(0)=0$ for every $\theta^I_j\in\Theta$, and  with the relation of partial order given~by
\begin{equation*}\label{def:modCont}
\theta^{I_1}_{j_1}\le\theta^{I_2}_{j_2}\quad \text{whenever } I_1\subseteq I_2 \text{ and } j_1\le j_2 \, .
\end{equation*}
\end{defn}
We can now introduce the family of sets $\TC\big(\R^M\big)$, where $\Theta$ is a suitable set of moduli of continuity.
\begin{defn}\label{def:TC}
Let   $\Theta$ be a suitable set of moduli of continuity, and $\K_j^I$ the set of functions in $C(I,B_j)$ which admit $\theta^I_j$ as modulus of continuity. We say that $f$ is \emph{$\Theta$\nbd-Carath\'eodory for $1\le p <\infty$} and write $f\in \TC\big(\R^M\big)$ (or simply $f\in\TC$ when $M=N$), if $f$ satisfies  (C1), (C2), and for each $j\in\N$ and $I=[q_1,q_2], \ q_1,q_2\in\Q$ it~holds
\begin{itemize}
\item[(T)] if $\big(x_n(\cdot)\big)_{\nin}$ is a sequence in $\K_j^I$  uniformly converging to $x(\cdot)\in\K_j^I$, then
\begin{equation}
\lim_{\nti}\int_I\big|f\big(t,x_n(t)\big)-f\big(t,x(t)\big)\big|^pdt=0.
\label{T}
\end{equation}
\end{itemize}
\end{defn}
\begin{rmk}
As regards Definitions \ref{def:LC}, \ref{def:SC} and \ref{def:TC}, we identify the functions which lay in the same set and only differ on a negligible subset of  $\R^{1+N}$. The constraint about belonging to the same set is crucial. Indeed, without any additional constraint, a function in $\SC\big(\R^M\big)$ could in fact be identified with a function which is not in $\SC\big(\R^M\big)$. Furthermore, such a rule implies that $\LC\big(\R^M\big)\subset\SC\big(\R^M\big)$ but $\SC\big(\R^M\big)$  is not included in $\TC\big(\R^M\big)$. Nevertheless, a continuous injection  (which is not a bijection, as we will show in section 6) of $\SC\big(\R^M\big)$ in $\TC\big(\R^M\big)$ is straightforward.  Thus, the following chain can be sketched
\begin{equation}
\LC\big(\R^M\big)\subset\SC\big(\R^M\big)\hookrightarrow\TC\big(\R^M\big)\, ,
\label{eq:SPincl}
\end{equation}
where $\Theta$ is any suitable set of moduli of continuity.
\end{rmk}
In particular, the following proposition characterizes the process of identification in $\TC\big(\R^M\big)$ and, as a consequence, implies that $\TC\big(\R^M\big)$ is a metric space when endowed with the topology defined immediately after.
\begin{prop}
Let $f,g\in\TC\big(\R^M\big)$ coincide almost everywhere in $\R\times\R^N$. Then, for any $\K_j^I$ as in \rm{Definition~\ref{def:TC}}, we have that
\begin{equation*}\label{prop:coincideAE=>CoincideOnTheContinuous}
\forall\, x(\cdot)\in\K_j^I\ :\ f\big(t,x(t)\big)=g\big(t,x(t)\big) \text{ for a.e. } t\in I\, .
\end{equation*}
\end{prop}
\begin{proof}
Consider $x(\cdot)\in\K_j^I$ and let $V\subset \R\times\R^N$ be such that
\begin{equation*}
f(t,x)=g(t,x)\quad \forall\, (t,x)\in V\quad \text{ and } \quad \meas_{\R^{1+N}}\left(\R^{1+N}\setminus V\right)=0 \, .
\end{equation*}
Consider the set $E=\left\{(t,\ep)\in I\times B_1\subset \R^{1+N}\mid \big(t, x(t)+\ep\big)\in V\right\}$, and for any $t\in I$ denote by $E_t$ the section in $t$  of $E$, i.e. $E_t=\{ \ep\in B_1\mid (t,\ep)\in E \}$. Now, for a given $t\in I$ one has
\begin{equation*}
x(t) + (B_1\setminus E_t)\subset B_{j+1} \setminus V_t\, .
\end{equation*}
Therefore, $\meas_{\R^N}(B_1\setminus E_t)=0$ for almost every $t\in I$. Then, applying Fubini's theorem twice, one has
\begin{equation*}
\meas_{\R}(I)\cdot \meas_{\R^N}(B_1)=\meas_{\R^{1+N}}(E)=\int_{\R^N}  \meas_{\R}(E_\ep) \, d\ep  \, ,
\end{equation*}
where $E_\ep$ denotes the section of $E$ given for any fixed $\ep\in B_1$. Therefore, we have that $\meas_{\R}(E_\ep)=\meas_{\R}(I)$ for almost every $\ep\in B_1$. Now, let $(\ep_n)_{\nin}\subset B_1$ be such that
 \begin{equation*}
\ep_n\xrightarrow{\nti}0\quad \text{ and }\quad \meas_{\R}(E_{\ep_n})=\meas_{\R}(I)\quad \forall \ n\in \N\,.
\end{equation*}
Then, called $x_n(t)=x(t)+\ep_n$ for any $\nin$, one has that $x_n(\cdot)\in\K_{j+1}^I$ and
\begin{equation*}
\begin{split}
\int_I\big|f\big(t,x(t)\big)&-g\big(t,x(t)\big)\big|\, dt\le \int_I\big|f\big(t,x(t)\big)-f\big(t,x_n(t)\big)\big|\, dt\\
&+ \int_I\big|f\big(t,x_n(t)\big)-g\big(t,x_n(t)\big)\big|\, dt+ \int_I\big|g\big(t,x_n(t)\big)-g\big(t,x(t)\big)\big|\, dt\,.
\end{split}
\end{equation*}
The terms on the right-hand side go to zero as $\nti$ because $f\big(t,x_n(t)\big)=g\big(t,x_n(t)\big)$ almost everywhere, $f$ and $g$ are in $\TC\big(\R^M\big)$ and $L^p_{loc}\subset L^1_{loc}$. Therefore, we have that $f\big(t,x(t)\big)=g\big(t,x(t)\big)$ almost everywhere.
\end{proof}
\par
Now we endow the previously introduced spaces with suitable topologies. As a rule, when inducing a topology on a subspace we will denote the induced topology with the same symbol which denotes the topology on the original space. The space $\TC\big(\R^M\big)$ will be endowed with the following topology.
\begin{defn}
Let $\Theta$ be a suitable set of moduli of continuity. We call $\T_{\Theta}$ the topology on $\TC\big(\R^M\big)$ generated by the family of seminorms
\begin{equation*}
p_{I,\, j}(f)=\sup_{x(\cdot)\in\K_j^I}\left[\int_I\big|f\big(t,x(t)\big)\big|^pdt\right]^{1/p} ,\quad f\in\TC\big(\R^M\big)\, ,
\end{equation*}
with $I=[q_1,q_2]$, $q_1,q_2\in\Q$, $j\in\N$, and $\K_j^I$  as in Definition~\ref{def:TC}.
$\left(\TC\big(\R^M\big),\T_{\Theta}\right)$ is a locally convex metric space.
\end{defn}
\par
We introduce two topologies on the set $\SC\big(\R^M\big)$.
\begin{defn}
We call $\T_{B}$ the topology on $\SC\big(\R^M\big)$ generated by the family of seminorms
\begin{equation*}
p_{I,\, j}(f)=\sup_{x(\cdot)\in C(I,B_j)}\left[\int_I\big|f\big(t,x(t)\big)\big|^pdt\right]^{1/p} ,\quad f\in\SC\big(\R^M\big)\, ,
\end{equation*}
where $I=[q_1,q_2]$, $q_1,q_2\in\Q$ and $j\in\N$.
$\left(\SC\big(\R^M\big),\T_{B}\right)$ is a locally convex metric space.
\end{defn}
\begin{defn}\label{def:TD}
Let $D$ be a countable and dense subset of $\R^N$. We call $\T_{D}$ the topology on $\SC\big(\R^M\big)$ generated by the family of seminorms
\begin{equation*}
p_{I,\, x_j}(f)=\left[\int_I|f(t,x_j)|^pdt\right]^p,\quad f\in\SC\big(\R^M\big), \, x_j\in D,\, I=[q_1,q_2],\,  q_1,q_2\in\Q\, .
\end{equation*}
$\left(\SC\big(\R^M\big),\T_{D}\right)$ is a  locally convex metric space.
\end{defn}
Notice that $\SC\big(\R^M\big)$ and $\LC\big(\R^M\big)$ can be endowed with all the previous topologies and the following chain of order holds
\begin{equation}
\T_D\le\T_{\Theta}\le\T_{B}\, .
\label{eq:TopIncl}
\end{equation}
We conclude this section presenting a result about the space $\left(\TC\big(\R^M\big),\T_{\Theta}\right)$.
\begin{thm}
Let $f$ be a function in $\C\big(\R^M\big)$. If there exists a sequence $(f_n)_{n\in\N}$  in
$ \TC\big(\R^M\big)$ such that for every $\K_j^I$, as in \rm{Definition~\ref{def:TC}}, one has
\begin{equation}\label{eq:13.09_10:10}
\lim_{\nti}\sup_{y(\cdot)\in\K_j^I}\int_I\big|f_{n}\big(t,y(t)\big)-f\big(t,y(t)\big)\big|^pdt=0\, ,
\end{equation}
then $f\in \TC\big(\R^M\big)$.
\label{thm:meas+conv=>ThetaC}
\end{thm}
\begin{proof}
Since condition (C1) and (C2) are satisfied by hypothesis, we only need to prove condition (T). Consider $I=[q_1,q_2]$, $q_1,q_2\in\Q$, $j\in\N$, and let $(x_k(\cdot))_{k\in\N}$ be a sequence in $\K_j^I$  converging uniformly to some  $x(\cdot)\in\K_j^I$. Thanks to equation \eqref{eq:13.09_10:10}, for a fixed $\ep>0$ there exists $n_0\in\N$ such that,
\begin{equation*}
\sup_{y(\cdot)\in\K_j^I}\left[\int_I\big|f_{n_0}\big(t,y(t)\big)-f\big(t,y(t)\big)\big|^pdt\right]^{1/p}<\ep/2 \, .
\end{equation*}
Therefore, we have that
\begin{equation}
\begin{split}
\big\|f\big(\cdot,x_k(\cdot)\big)-&f\big(\cdot,x(\cdot)\big)\big\|_{p}\le\normp{f\big(\cdot,x_k(\cdot)\big)-f_{n_0}\big(\cdot,x_k(\cdot)\big)}{I}\\
+&\normp{f_{n_0}\big(\cdot,x_k(\cdot)\big)-f_{n_0}\big(\cdot,x(\cdot)\big)}{I} +\normp{f\big(\cdot,x(\cdot)\big)-f_{n_0}\big(\cdot,x(\cdot)\big)}{I}\\
\le\,\ep\, &+\normp{f_{n_0}\big(\cdot,x_k(\cdot)\big)-f_{n_0}\big(\cdot,x(\cdot)\big)}{I}.
\end{split}
\label{eq:03.05_13:15}
\end{equation}
Then, recalling that $f_{n_0}\in\TC\big(\R^M\big)$, from~\eqref{T} and \eqref{eq:03.05_13:15}, we conclude that
\begin{equation*}
\lim_{k\to\infty}\int_I\big|f\big(t,x_k(t)\big)-f\big(t,x(t)\big)\big|^pdt=0\,,
\end{equation*}
and condition (T) holds for $f$.
\end{proof}
\section{Continuity of time translations}\label{sectrans}
Let $\Theta$ be a suitable set of moduli of continuity as in Definition~\ref{def:modCont} and consider $f\in\TC\big(\R^M\big)$. In the following we will denote by $f_t$ the time translation of $f$, i.e. the map of $\R\times\R^N$ into $\R^M$ defined by $(s,x)\mapsto f_t(s,x)=f(s+t,x)$, where trivially $f_t\in\TC\big(\R^M\big)$ for every $t\in\R$. The aim of this section is to prove that the time translation defines a continuous flow on $\left(\TC\big(\R^M\big),\T_\Theta\right)$.
\begin{thm}\label{thm:theta_translCont}
Let $\Theta$ be a suitable set of moduli of continuity. The map
\begin{equation*}
\Pi:\R\times \TC\big(\R^M\big)\to \TC\big(\R^M\big)\, ,\qquad (t,f)\mapsto\Pi(t,f)=f_t\, ,
\end{equation*}
defines a continuous flow on $\left(\TC\big(\R^M\big),\T_\Theta\right)$.
\label{thm:continuityBASE}
\end{thm}
\begin{proof}
We separately deal with the continuity with respect to $t$ and with respect to $f$, and eventually gather them together.
\par\vspace{0.05cm}
Let $(f_n)_{\nin}$  be a sequence in $\TC\big(\R^M\big)$  converging to some $f$ in $\left(\TC\big(\R^M\big),\T_{\Theta}\right)$. We prove that $(f_n)_t\to f_t$ in $\left(\TC\big(\R^M\big),\T_{\Theta}\right)$ as $\nti$, uniformly for $t$ in a compact interval.
Consider $I=[p_1,p_2]$ and $J=[q_1,q_2]$ such that  $p_1,p_2,q_1,q_2\in\Q$, $0\in J$ and fix $t\in J$. Moreover, for any $j\in\N$ consider $\K_j^I$ and $\K_j^{I+J}$ as in Definition~\ref{def:TC}. Notice that $x(\cdot)\in\K_j^I$ implies $x(\cdot-t)\in\K_j^{I+J}$ up to a suitable extension by constants of the function $x(\cdot-t)$ in $I+J$. Then
\begin{equation}\label{eq:07.07-12:02}
\begin{split}
\lim_{\nti}\sup_{x(\cdot)\in\K_j^I}\int_I\big|(f_n&)_t\big(s,x(s\big))-f_t\big(s,x(s)\big)\big|^pds\\
&=\lim_{\nti}\sup_{x(\cdot)\in\K_j^I}\int_{I+t}\big|f_n\big(r,x(r-t)\big)-f\big(r,x(r-t)\big)\big|^pdr\\
&\le\lim_{\nti}\sup_{x(\cdot)\in\K_j^{I+J}}\int_{I+J}\big|f_n\big(r,x(r)\big)-f\big(r,x(r)\big)\big|^pdr=0\, .
\end{split}
\end{equation}
\par\vspace{0.05cm}
Next, we prove the continuity with respect to the first variable; in other words, the map $t\mapsto f_t$ of $\R$ into $\left(\TC\big(\R^M\big),\T_{\Theta}\right)$ is continuous.
Consider $f\in\TC\big(\R^M\big)$, $I=[a,b]$ where  $a,b\in\Q$ and $t\in\R$ fixed. We aim to prove that for any compact set $\K_j^I$, as  in Definition~\ref{def:TC}, we have that
\begin{equation}
\lim_{\tau\to0}\sup_{x(\cdot)\in\K_j^I}\int_I\big|f_{t+\tau}\big(s,x(s)\big)-f_t\big(s,x(s)\big)\big|^pds=0\,.
\label{eq:Lem2th}
\end{equation}
Firstly, let us fix $x(\cdot)\in\K_j^I$ and prove that if $\tau_n\to0$ as $\nti$ then
\begin{equation}
\lim_{\nti}\int_I\big|f_{t+\tau_n}\big(s,x(s)\big)-f_t\big(s,x(s)\big)\big|^pds=0\, .
\label{eq:06.07-13:48}
\end{equation}
Notice that $f_t\big(\cdot,x(\cdot)\big)\in L^p\big(I,\R^M\big)$ and consider the operator $T_\tau: L^p\big(I,\R^M\big)\to L^p\big(\R,\R^M\big)$, such that $g(\cdot)\mapsto T_\tau g(\cdot)$, where $T_\tau g(\cdot)$ is defined by
\begin{equation*}
T_{\tau}g(s)=
\begin{cases}
g(s+\tau)\, , & \text{if }s+\tau \in I\\
 \quad\ \ \, 0\, , & \text{otherwise.}
\end{cases}
\end{equation*}
By the continuity of translations in $L^p(I)$, see Castillo and Rafeiro \cite[Theorem 3.58]{book:CR}, we have that, if $|\tau_n|\to 0$ as $\nti$, then for a given $\ep>0$ there exists $\delta>0$ such that $\sup_{|\tau_n|<\delta}\normp{T_{\tau_n}f_t\big(\cdot, x(\cdot)\big)-f_t\big(\cdot, x(\cdot)\big)}{I}\le \ep$.
\par
Now, for any $\nin$ define $a_n=\max\{a,a-\tau_n\}$ and $b_n=\min\{b,b-\tau_n\}$, and  consider $n_0\in\N$ so that for any $n>n_0$ we have $|\tau_n|<\delta$. Therefore, for any $n>n_0$ the following chain of inequalities holds
\begin{equation*}
\begin{split}
&\normp{f_{t+\tau_n}\big(\cdot,x(\cdot)\big)-f_t\big(\cdot,x(\cdot)\big)}{I}\\
&\qquad\le \normp{T_{\tau_n}f_t\big(\cdot,x(\cdot)\big)-f_t\big(\cdot,x(\cdot)\big)}{I}+\normp{f_{t+\tau_n}\big(\cdot,x(\cdot)\big)-T_{\tau_n}f_t\big(\cdot,x(\cdot)\big)}{I}\\
&\qquad\le \, \ep +\normp{f_{t+\tau_n}\big(\cdot,x(\cdot)\big)-T_{\tau_n}f_t\big(\cdot,x(\cdot)\big)}{I}\\
&\qquad\le \, \ep +\bigg[\int_{a_n}^{b_n}\big|f_{t}\big(s+\tau_n,x(s)\big)-f_t\big(s+\tau_n,x(s+\tau_n)\big)\big|^pds\bigg]^{1/p}\\
&\qquad\qquad+\left[\int_{a}^{a_n}\big|f_{t}\big(s+\tau_n,x(s)\big)\big|^pds\right]^{1/p}+\bigg[\int_{b_n}^{b}\big|f_{t}\big(s+\tau_n,x(s)\big)\big|^pds\bigg]^{1/p}\\
&\qquad\le \, \ep +\bigg[\int_{a_n+\tau_n}^{b_n+\tau_n}\big|f_{t}\big(u,x(u-\tau_n)\big)-f_t\big(u,x(u)\big)\big|^pdu\bigg]^{1/p}\\
&\qquad\qquad+\left[\int_{a+\tau_n}^{a_n+\tau_n}\big|f_{t}\big(u,x(u-\tau_n)\big)\big|^pdu\right]^{1/p}+\bigg[\int_{b_n+\tau_n}^{b+\tau_n}\big|f_{t}\big(u,x(u-\tau_n)\big)\big|^pdu\bigg]^{1/p}\\
&\qquad= \, \ep +I_1+I_2+I_3\, ,
\end{split}
\end{equation*}
As regard $I_1$, notice that, up to extending the functions $x(\cdot)$ and $\big(x(\cdot -\tau_n)\big)_\nin$ by constants to an interval $J$ containing $I+[-\delta,\delta]$ we have that
\begin{equation}
I_1\le \left[\int_J\big|f_{t}\big(u,x(u-\tau_n)\big)-f_t\big(u,x(u)\big)\big|^pdu\right]^{1/p}\,,
\label{eq:06.07-13:30}
\end{equation}
and the integral on the right-hand side of equation \eqref{eq:06.07-13:30} goes to zero as $\nti$, due to the fact that $f\in\TC\big(\R^M\big)$ and $\|x(\cdot -\tau_n)-x(\cdot)\|_{\infty}\to0$ in $J$ as $\nti$. As regard  $I_2$, let $m^j$ be an $m$-bound of $f$ on $B_j$ and notice that the following chain of inequalities holds
\begin{equation}
I_2\le \bigg[\int_{a-|\tau_n|}^{a}\big|f_{t}\big(u,x(u-\tau_n)\big)\big|^pdu\bigg]^{1/p}\le \bigg[\int_{a-|\tau_n|}^{a}\big(m^j_{t}(u)\big)^pdu\bigg]^{1/p}\,,
\label{eq:06.07-13:43}
\end{equation}
and the integral on the right-hand side of equation \eqref{eq:06.07-13:43} goes to zero as $\nti$, thanks to the absolute continuity of the Lebesgue integral. Similar reasonings apply to $I_3$. Therefore, for any fixed $t\in \R$ and $x(\cdot)\in\K_j^I$ we obtain the limit in \eqref{eq:06.07-13:48}. Next we check that such a convergence is uniform in $\K_j^I$. Otherwise there would exist an $\ep>0$, a sequence $\big(x_n(\cdot)\big)_{n\in\N}$  in $\K_j^I$, and a sequence $(\tau_n)_{n\in\N}$ in $\R$ converging to $0$, such that
\begin{equation*}
\left[\int_I\big|f_{t+\tau_n}\big(s,x_n(s)\big)-f_t\big(s,x_n(s)\big)\big|^pds\right]^{1/p}>\ep, \qquad \forall\ n\in\N\,.
\end{equation*}
However, being $\K_j^I$  compact, there exists a convergent subsequence of $\big(x_n(\cdot)\big)_{n\in\N}$, which we keep on denoting with the same indexes,  converging uniformly in $I$ to some $x(\cdot)\in\K_j^I$ as $n\to\infty$. Nevertheless, from \eqref{eq:06.07-13:48},  there exists $n_0\in\N$ such that, if $n>n_0$, then
\begin{equation}
\normp{f_{t+\tau_n}\big(\cdot,x(\cdot)\big)-f_t\big(\cdot,x(\cdot)\big)}{I}<\frac{\ep}{4}\,.
\label{eq:06.07-18:22}
\end{equation}
Moreover, since $f_t\in\TC\big(\R^M\big)$ and $\big(x_n(\cdot)\big)_{n\in\N}$ converges uniformly to $x(\cdot)$, there exists $n_1\in\N$ such that, if $n>n_1$, then
\begin{equation}
\normp{f_{t}\big(\cdot,x(\cdot)\big)-f_{t}\big(\cdot,x_n(\cdot)\big)}{I}<\frac{\ep}{4}\, .
\label{eq:07.07-11:49}
\end{equation}
Then, for $n> \max\{n_0,n_1\}$, we have that
\begin{equation}
\begin{split}
\ep&<\normp{f_{t+\tau_n}\big(\cdot,x_n(\cdot)\big)-f_t\big(\cdot,x_n(\cdot)\big)}{I}\\
&\le\normp{f_{t+\tau_n}\big(\cdot,x_n(\cdot)\big)-f_{t+\tau_n}\big(\cdot,x(\cdot)\big)}{I}+\normp{f_{t+\tau_n}\big(\cdot,x(\cdot)\big)-f_t\big(\cdot,x(\cdot)\big)}{I}\\
&\qquad \quad +\normp{f_{t}\big(\cdot,x(\cdot)\big)-f_{t}\big(\cdot,x_n(\cdot)\big)}{I}\\
&=A_1+\ep/4+\ep/4
\end{split}
\label{eq:07.07-11:46}
\end{equation}
Finally, notice that
\begin{equation}
\begin{split}
A_1&=\left[\int_{I+\tau_n}\big|f_{t}\big(u,x_n(u-\tau_n)\big)-f_t\big(u,x(u-\tau_n)\big)\big|^pdu\right]^{1/p}\\
&\le \left[\int_{J}\big|f_{t}\big(u,x_n(u-\tau_n)\big)-f_t\big(u,x(u-\tau_n)\big)\big|^pdu\right]^{1/p}<\frac{\ep}{4}\, ,
\end{split}
\label{eq:07.07-11:47}
\end{equation}
for $n$ greater than some $n_2\in\N$ since, once again, $f_t\in\TC\big(\R^M\big)$ and $\big(x_n(\cdot)\big)_{n\in\N}$ converges uniformly to $x(\cdot)$.
Gathering \eqref{eq:07.07-11:46}, \eqref{eq:06.07-18:22}, \eqref{eq:07.07-11:49} and \eqref{eq:07.07-11:47} we get a contradiction, which implies the uniform limit in \eqref{eq:Lem2th}.
\par\vspace{0.05cm}
In order to conclude the proof, consider $(f_n)_{\nin}\subset\TC\big(\R^M\big)$ converging to some $f$ in $\left(\TC\big(\R^M\big),\T_{\Theta}\right)$ and $(t_n)_{\nin}\subset\R$  converging to some $t\in\R$. Fixed $j\in\N$, $I=[q_1,q_2]$ where $q_1,q_2\in\Q$,  and $\K_j^I$ as in Definition~\ref{def:TC}, recalling that the limit in \eqref{eq:07.07-12:02} is uniform for $t$ in compact sets, we have that
\begin{equation*}
\begin{split}
\lim_{\nti}\sup_{x(\cdot)\in\K_j^I}&\left[\int_I\big|(f_n)_{t_n}\big(s,x(s)\big)-f_t\big(s,x(s)\big)\big|^pds\right]^{1/p}\\
&\le \lim_{\nti}\ \sup_{x(\cdot)\in\K_j^I}\left[\int_I\big|(f_n)_{t_n}\big(s,x(s)\big)-f_{t_n}\big(s,x(s)\big)\big|^pds\right]^{1/p}\\
&\quad \ \ \, +\lim_{\nti}\sup_{x(\cdot)\in\K_j^I}\left[\int_I\big|f_{t_n}\big(s,x(s)\big)-f_t\big(s,x(s)\big)\big|^pds\right]^{1/p}=0\,,
\end{split}
\end{equation*}
which ends the proof.
\end{proof}
\begin{rmk}
The continuity of the time translation map in $\left(\SC\big(\R^M\big),\T_{D}\right)$ can be easily proved using the same arguments of the proof of Theorem \ref{thm:continuityBASE}. Therefore, the proof is omitted. Nevertheless, the continuity in $\left(\SC\big(\R^M\big),\T_{B}\right)$ is stated in \cite{book:GS} and the proof can be derived by the one given for $\left(\LC\big(\R^M\big),\T_{B}\right)$ in \cite{book:RMGS}.
\end{rmk}
We conclude this section introducing the concept of hull of a function.
\begin{defn}\label{def:Hull}
Let $E$ denote one of the sets in \eqref{eq:SPincl} and $\T$  one of the topologies in \eqref{eq:TopIncl}, assuming that endowing $E$ with the topology $\T$ makes sense. If $f\in E$, we call  \emph{the hull of $f$ with respect to $(E,\T)$}, the topological subspace of  $(E,\T)$ defined~by
\begin{equation*}
\mathrm{Hull}_{(E,\T)}(f)=\big(\mathrm{cls}_{(E,\T)}\{f_t\mid t\in\R\} ,\, \T\big)\, ,
\end{equation*}
where, $\mathrm{cls}_{(E,\T)}(A)$ represents the closure in $(E,\T)$ of the set $A$, and $\T$ is the induced topology.
\end{defn}
As a corollary of the previous theorem, we deduce  the continuity of the translations in any suitable hull.
\begin{cor}
Let $(E,\T)$  be defined as in {\rm Definition~\ref{def:Hull}}, and let $f$ be a function in $E$. Then, the map
\begin{equation*}
\Pi:\R\times \mathrm{Hull}_{(E,\T)}(f)\to \mathrm{Hull}_{(E,\T)}(f)\, ,\qquad (t,g)\mapsto\Pi(t,g)=g_t\, ,
\end{equation*}
defines a continuous flow on $\mathrm{Hull}_{(E,\T)}(f)$.
\end{cor}
\section{Topological properties of the $m$-bounds and $l$-bounds}\label{secTopml}
In this section we will analyze some topological properties of Carath\'eodory functions admitting $L^1_{loc}$-equicontinuous $m$-bounds and/or $L^p_{loc}$-bounded $l$-bounds. The role of $m$-bounds and/or $l$-bounds in proving the continuous variation of ODEs' solutions with respect to initial conditions, has been fully explored in \cite{paper:ZA1,paper:ZA2} when weak topologies are involved. As a matter of fact, section \ref{secContFlowODEs} will show how such topological properties turn out to be useful in order to prove the continuity when the strong topologies introduced in section \ref{sectopo} are employed.\par\vspace{0.05cm}
We start recalling that a subset $S$ of positive functions in $L^p_{loc}$  is bounded if for every $r>0$ the following inequality holds
\begin{equation*}
\sup_{m\in S}\int_{-r}^r m^p(t)\,  dt<\infty\,.
\end{equation*}
In such a case we will say that $S$ is $L^p_{loc}$-bounded.
\begin{defn}\label{weakcomp}
A set $S$ of positive functions in $L^1_{loc}$ \emph{is $L^1_{loc}$-equicontinuous} if for any $r>0$ and for any $\ep>0$ there exists a $\delta=\delta(r,\ep)>0$ such that, for any $-r\le s\le t\le r$, with $t-s<\delta$, we have
\begin{equation*}
\sup_{m\in S}\int_{s}^t m(u)\,du<\ep\, .
\end{equation*}
\end{defn}
\begin{rmk}\label{rmk:equicnt=>bound}
According to the previous definitions,  the $L^1_{loc}$-equicontinuity implies the $L^1_{loc}$-boundedness. On the other hand, if $p>1$ the $L^p_{loc}$-boundedness implies  the $L^1_{loc}$-equicontinuity.
\end{rmk}
In the following, let $\M^+$ be the set of  positive and regular Borel measures on $\R$. We endow $\M^+$ with the following topology.
\begin{defn}
We say that  \emph{a sequence  $(\mu_n)_\nin$ of measures in $\M^+$ vaguely converges to $\mu\in\M^+$}, and write $\mu_n\xrightarrow[]{\widetilde{\sigma}}\mu$, if and only if
\begin{equation*}
 \lim_{\nti}\int_\R \phi(s)\, d\mu_n(s)=\int_\R\phi(s)\, d\mu(s)\qquad \text{for each } \phi\in C^+_C(\R).
\end{equation*}
We will denote such a topological space by $(\M^+,\widetilde{\sigma})$.
\end{defn}
As shown in  Kallenberg \cite[Theorem 15.7.7, p.170]{book:Kall}, $(\M^+,\widetilde{\sigma})$ is a Polish space, i.e. it is separable and completely metrizable. Moreover, the following proposition holds (see \cite[Theorem 15.7.5, p.170]{book:Kall}).
\begin{prop}\label{prop:boundedness}
Any subset $M_0\subset\M^+$ is relatively compact in the vague topology if and only if $\sup_{\mu\in M_0}\mu (B)<\infty$ for any bounded Borel set $B$.
\end{prop}
Easy arguments of measure theory allow to prove the following characterization of the $L^1_{loc}$-equicontinuous subsets of positive functions in $L^1_{loc}$ through the relative compactness of the associated set of measures in $\M^+$. In order to proceed with the statement, we need to set some notation. We denote by $\M^+_{ac}$ the set of measures $\mu\in\M^+$ such that for every $r\in\R^+$ the restriction of $\mu$ to the interval $[-r,r]$, namely $\mu\lfloor_{[-r,r]}$ is absolutely continuous with respect to the Lebesgue measure. The sets $\M^+_{sc}$ and $\M^+_{pd}$ of singular continuous and purely discontinuous measures respectively, can be similarly defined. Trivially, $\M^+=\M^+_{ac}\oplus\M^+_{sc}\oplus\M^+_{pd}$.
\begin{thm}
Let $S\subset L^1_{loc}$ be a set of positive functions and let $M\subset\M^+_{ac}$ be the set of absolutely continuous measures whose densities are the functions of $S$. Then, the following statements are equivalent.
\begin{itemize}
\item[(i)] $S$ is $L^1_{loc}$-equicontinuous.
\item[(ii)] $M$ is relatively compact in $(\M^+, \widetilde{\sigma})$ and $\mathrm{cls}_{(\M^+, \widetilde{\sigma})}(M)\subset \M^+_{ac}\oplus\M^+_{sc}$.
\end{itemize}
\end{thm}
The following definition extends the previous concepts to sets of Carath\'eodory functions through their $m$-bounds and/or $l$-bounds.
\begin{defn}
We say that
\begin{itemize}
\item[(i)] a set $E\subset\SC\big(\R^M\big)$  \emph{admits $L^p_{loc}$-bounded (resp. $L^1_{loc}$-equicontinuous)} $m$-bounds, if for any $j\in\N$, the set $S^j\subset L^p_{loc}$, made up of the optimal $m$-bounds on $B_j$ of the functions in $E$, is $L^p_{loc}$\nbd-bounded (resp. $L^1_{loc}$-equicontinuous);
\item[(ii)] \emph{$f\in\SC\big(\R^M\big)$ admits $L^p_{loc}$-bounded (resp. $L^1_{loc}$-equicontinuous) $m$-bounds} if the set $\{f_t\mid t\in\R\}$ admits  $L^p_{loc}$-bounded (resp. $L^1_{loc}$-equicontinuous) $m$-bounds;
\item[(iii)]a set $E\subset\LC\big(\R^M\big)$  \emph{has  $L^p_{loc}$-bounded  (resp. $L^1_{loc}$-equicontinuous)} $l$\nbd-bounds, if for any $j\in\N$, the set $S^j\subset L^p_{loc}$, made up of the optimal $l$-bounds on $B_j$  of the functions in $E$,  is  $L^p_{loc}$-bounded  (resp. $L^1_{loc}$-equicontinuous);
\item[(iv)] $f\in\LC\big(\R^M\big)$  \emph{has $L^p_{loc}$-bounded  (resp. $L^1_{loc}$-equicontinuous) $l$-bounds} if  the set $\{f_t\mid t\in\R\}$  has $L^p_{loc}$-bounded (resp. $L^1_{loc}$-equicontinuous) $l$-bounds.
\end{itemize}
\label{def:05.07-13:05}
\end{defn}
\begin{prop}
Let $E$ be a set of functions in $\SC\big(\R^M\big)$  admitting $L^p_{loc}$-bounded $m$-bounds, $D$ a countable and dense subset of $\R^N$, and  $f\colon \R\times\R^N\to \R^M$  a Borel function such that, for almost every $t\in\R$, $f$ is continuous in $x$. If $(f_n)_\nin$ is a sequence in $E$ such that for any $x_j\in D$ one has $f_n(\cdot,x_j)\to f(\cdot,x_j)$ in $L^p_{loc}$, then $f\in\SC\big(\R^M\big)$.
\label{prop:mbounds}
\end{prop}
\begin{proof}
Firstly, let us work with $p=1$. Fix $j\in\N$ and, for any $\nin$, let $m^j_n$ be the optimal $m$-bound of $f_n$ on $B_j$ and  $\mu^j_n\in\M^+$ be the positive absolutely continuous measure (with respect to Lebesgue measure) with density $m^j_n(\cdot)$. By hypothesis, the set $\{ m^j_n(\cdot)\mid\nin\}$ is $L^1_{loc}$-bounded. Hence, due to Proposition \ref{prop:boundedness}, the sequence of induced measures $(\mu^j_n)_\nin$,  is relatively compact in $(\M^+,\widetilde{\sigma})$ and thus vaguely converges, up to a subsequence, to a measure $\mu^j\in\M^+$. Moreover, by Lebesgue-Besicovitch differentiation theorem, there exists $m^j(\cdot)\in L^1_{loc}$ such that
\begin{equation}
m^j(t)=\lim_{h\to0}\frac{\mu^j([t,t+h])}{h}\, , \qquad \mathrm{for\ a.e.} \ t\in\R\, ,
\label{eq:30.05-13:30}
\end{equation}
and $m^j(\cdot)$ is the density of the absolutely continuous part of the Radon-Nikod\'ym decomposition of $\mu^j$ in each compact interval. We claim that $m^j(\cdot)$ is an $m$-bound for $f$ on $B_j$. Indeed, fixed $x\in D\cap B_j$, $t,h\in\R$, and considered $\phi\in C_C^+(\R)$ such that $\phi\equiv 1$ in $[t,t+h]$, we have
\begin{equation*}
\begin{split}
\frac{1}{h}\int_{t}^{t+h}|f(s,x)|\, ds&=\lim_{\nti}\frac{1}{h}\int_{t}^{t+h}|f_n(s,x)|\, ds\\
&\le \lim_{\nti}\frac{1}{h}\int_\R\phi(s)\, d\mu^j_n(s)=\frac{1}{h}\int_\R\phi(s)\, d\mu^j(s)\,.
\end{split}
\end{equation*}
Moreover, by the regularity of $\mu^j$
\begin{equation*}
 \mu^j([t,t+h])=\inf\left\{\int_\R\phi(s)\, d\mu^j(s)\;\Big|\; \phi\in C_C^+(\R),\;  \phi\equiv 1\; \text{in } [t,t+h]\right\}\, ,
\end{equation*}
then we have that
\begin{equation*}
\frac{1}{h}\int_{t}^{t+h}|f(s,x)|\, ds\le\frac{\mu^j([t,t+h])}{h}\,.
\end{equation*}
Thus, passing to the limit as $h\to0$, we obtain the aimed inequality for almost every $t\in\R$ for the fixed $x\in D\cap B_j$. Now, in order to obtain the result on the entire $B_j$, consider $T\subset \R$ of full measure be such that $|f(t,x)|\le m^j(t)$ for every $t\in T$ and for every $x\in D\cap B_j$. Then, by the continuity of $f(t,\cdot)$, we obtain the result for almost every $t\in\R$ for all $x\in B_j$, and $m^j$ provides an $m$-bound for $f$ in $B_j$, as claimed.
\par\vspace{0.05cm}
The same reasonings apply for $p>1$, recalling that $L^p_{loc}\subset L^1_{loc}$, and we only need to prove that the function $m^j(\cdot)\in L^1_{loc}$, provided by \eqref {eq:30.05-13:30}, is also in $L^p_{loc}$. By hypothesis  $\{ m^j_n(\cdot)\mid\nin\}$ is $L^p_{loc}$-bounded and, by Alaoglu-Bourbaki theorem, for every $r>0$ the closed balls of $L^p([-r,r])$ are relatively compact in the weak topology $\sigma\big(L^p([-r,r]),L^q([-r,r])\big)$. Therefore, if  $\big( m^j_{i_n}(\cdot)\big)_\nin$ is a weakly convergent subsequence of $\big( m^j_{n}(\cdot)\big)_\nin$ with limit $m^*(\cdot)\in L^p([-r,r])$, then the sequence of induced measures $(\mu^j_{i_n})_\nin$ vaguely converges to the absolutely continuous measure whose density is $m^*(\cdot)$ in $[-r,r]$. Hence, since Equation \eqref {eq:30.05-13:30} holds, $m^*(\cdot) $ has to coincide with $m^j(\cdot)$ in $[-r,r]$.
\end{proof}
Similar arguments allow to prove the following result.
\begin{prop}
Let  $E\subset \LC\big(\R^M\big)$  admit $L^p_{loc}$-bounded $l$-bounds and $D$ be a countable and dense subset of $\R^N$. Then, the $\mathrm{cls}_{(\SC\left(\R^M\right),\T_D)}(E)$ is in $\LC\big(\R^M\big)$.
\label{prop:25.05-10:53}
\end{prop}
\begin{proof}
Consider a sequence $(f_n)_\nin$ of functions in $ E$ converging to some $f\in \SC\big(\R^M\big) $ in $\left(\SC\big(\R^M\big),\T_D\right)$. Fix $j\in\N$ and, for any $\nin$, let $l^j_n$ be the optimal $l$-bound of $f_n$ on $B_j$. Reasoning like in the proof of Proposition~\ref{prop:mbounds}, we have that the sequence of absolutely continuous measures with densities  $(l^j_n(t))_\nin$ vaguely converges, up to a subsequence, to a positive measure whose absolutely continuous part has a function $l^j(\cdot)\in L^p_{loc}$ as density. Additionally, for any $x,y\in D\cap B_j$ with $x\neq y$ the following inequality holds
\begin{equation*}
|f(t,x)-f(t,y)|\le l^j(t)\,|x-y|\quad \text{for a.e. } t\in\R\,.
\end{equation*}
An extension of the previous inequality to the entire $B_j$ is thus achieved by continuity,  like in Proposition~\ref{prop:mbounds}. Therefore, $f\in\LC\big(\R^M\big)$, which ends the proof.
\end{proof}
As a consequence we deduce the following property for the hull of a $\LC\big(\R^M\big)$ function in different topologies.
\begin{cor}\label{cor:LP}
Let $f$ be a function in $\LC\big(\R^M\big)$ with $L^p_{loc}$-bounded $l$-bounds, and $\T$ be any of the introduced topologies. Then, we have that $\mathrm{Hull}_{(\SC(\R^M),\T)}(f)\subset\LC\big(\R^M\big)$ and $\mathrm{Hull}_{(\SC(\R^M),\T)}(f)=\mathrm{Hull}_{(\LC(\R^M),\T)}(f)$.
\end{cor}
The next result proves that the existence of $L^p_{loc}$-bounded or $L^1_{loc}$-equicontinuous $m$-bounds and/or $l$-bounds for a set $E\subset \SC\big(\R^M\big)$ is inherited by all the elements of the closure of $E$ with respect to any of the previously introduced topologies.
\begin{prop} Let $\T$ be any of the introduced topologies.
\begin{itemize}
\item[(i)] If $E\subset \SC\big(\R^M\big)$ (resp. $E\subset\LC\big(\R^M\big)$) admits $L^p_{loc}$-bounded $m$-bounds  (resp. $L^p_{loc}$-bounded $l$-bounds)  then $\mathrm{cls}_{(\SC(\R^M),\T)}(E)$ has $L^p_{loc}$-bounded $m$-bounds  (resp. $L^p_{loc}$-bounded $l$-bounds).
\item[(ii)]
 If $E\subset \SC\big(\R^M\big)$ (resp. $E\subset\LC\big(\R^M\big)$) admits $L^1_{loc}$-equicontinuous $m$-bounds  (resp. $L^1_{loc}$-equicontinuous  $l$-bounds), then $\mathrm{cls}_{(\SC(\R^M),\T)}(E)$   has $L^1_{loc}$-equicontinuous $m$-bounds  (resp.  $L^1_{loc}$-equicontinuous $l$-bounds).
 \end{itemize}
\label{prop:07.07-19:44}
\end{prop}
\begin{proof}Let us firstly assume that $p=1$ and that $E\subset \SC$ admits $L^1_{loc}$-bounded $m$-bounds, i.e. for every $r>0$ and every $j\in\N$ we have
\begin{equation*}
\sup_{f\in E}\int_{-r}^rm_f^j(t)\,dt<\infty\,,
\end{equation*}
where $m_f^j(\cdot)$ denotes the optimal $m$-bound of $f$ on $B_j$.  Let us denote by $\overline E=\mathrm{cls}_{(\mathfrak{SC}_1(\R^M),\T)}(E)$, and for any $g\in \overline E$ denote by $m_g^j(\cdot)$ either, the optimal $m$-bound of $g$ on $B_j$, if $g\in E$, or the $m$-bound of $g$ given by Proposition~\ref{prop:mbounds} if $g\in \overline E\setminus E$, i.e. the absolutely continuous part (with respect to Lebesgue measure) of the limit measure. Moreover, for any $g\in \overline E$, let $(g_n)_\nin$ be a sequence in $E$ converging to $g$ in $\left(\mathfrak{SC}_1\big(\R^M\big),\T\right)$.
Consider $r>0$ and $\phi\in C^+_C$ such that $\supp \phi\subset [-r-1,r+1]$ and $\phi\equiv 1$ in $[-r,r]$. Then,
\begin{equation*}
\begin{split}
\sup_{g\in\overline E}\int_{-r}^r m_g^j(t)\, dt&\le\sup_{g\in\overline E}\int_\R\phi(t)\,  m_g^j(t)\, dt\le\sup_{g\in\overline E}\, \lim_\nti \int_\R \phi(t)\, m_{g_n}^j(t)\, dt\\
&\le\sup_{g\in\overline E}\sup_{\nin}\int_{-r-1}^{r+1} \, m_{g_n}^j(t)\, dt\le\sup_{f\in E}\int_{-r-1}^{r+1} \, m_{f}^j(t)\, dt<\infty\, .
\end{split}
\end{equation*}
Therefore, $\overline E$ admits $L^1_{loc}$-bounded $m$-bounds. Analogous reasonings apply to the rest of the cases in (i) and (ii).\par\vspace{0.05cm}
If $p>1$ the result is a consequence of the weak relative compactness of the closed balls in $L^p([-r,r])$ for every $r>0$, where we employ the same reasonings used in the last part of the proof of Proposition \ref{prop:mbounds}.
\end{proof}
\begin{cor} Let $\T$ be any of the introduced topologies.
\begin{itemize}
\item[(i)] If $f\in \SC\big(\R^M\big)$ (resp. $\LC\big(\R^M\big)$) has $L^p_{loc}$-bounded $m$-bounds  (resp. $L^p_{loc}$-bounded $l$-bounds)  then any $g\in\mathrm{Hull}_{(\SC(\R^M),\T)}(f)$ has $L^p_{loc}$-bounded $m$\nbd-bounds  (resp. $L^p_{loc}$-bounded $l$-bounds).
\item[(ii)]
 If $f\in \SC\big(\R^M\big)$ (resp. $\LC\big(\R^M\big)$) has $L^1_{loc}$-equicontinuous $m$-bounds  (resp. $L^1_{loc}$-equicontinuous  $l$-bounds), then any $g\in\mathrm{Hull}_{(\SC(\R^M),\T)}(f)$  has $L^1_{loc}$-equicontinuous $m$-bounds  (resp.  $L^1_{loc}$-equicontinuous $l$-bounds).
 \end{itemize}
\label{prop:25.05-18:53}
\end{cor}
As we have noticed before, all the introduced topologies can be induced on $\LC\big(\R^M\big)$, where, on suitable subsets, they coincide as shown in the following result.
\begin{thm}
Let $E$ be a set in $\LC\big(\R^M\big)$ with $L^p_{loc}$-bounded $l$-bounds, then
\begin{equation*}
(E,\T_1)=(E,\T_2)\qquad\text{and}\qquad\mathrm{cls}_{(\SC(\R^M),\T_1)}(E)=\mathrm{cls}_{(\SC(\R^M),\T_2)}(E)\, ,
\end{equation*}
where $\T_1$ and $\T_2$ are any of the previously introduced topologies.
\label{thm:22.07-11:24}
\end{thm}
\begin{proof}
From Proposition~\ref{prop:25.05-10:53} we know that $\mathrm{cls}_{(\SC(\R^M),\T)}(E)\subset\LC$ for any $\T$. Moreover, due to relation \eqref{eq:TopIncl}, it suffices to prove that if $(f_n)_\nin$ is a sequence of elements of $ E$ converging to some $f$ in $\big(\LC\big(\R^M\big),\T_{D}\big)$, then $(f_n)_\nin$ converges to $f$ in $\big(\LC\big(\R^M\big),\T_{B}\big)$.  Fix a compact interval  $I=[q_1,q_2]$, with $q_1,q_2\in\Q$, $j\in\N$ and, for any $\nin$, let $l_n^j(\cdot)\in L^p_{loc}$ be the optimal $l$-bound of $f_n$ on $B_j$, as in Definition~\ref{def:LC}. By hypothesis, there exists $\rho>0$ such that
\begin{equation*}
\sup_{\nin}\int_I\big(l^j_n(s)\big)^p\, ds<\rho<\infty\, .
\end{equation*}
Now, fix $\ep>0$ and consider $\delta=\ep/(3\rho^{1/p})$. Since $ B_j \subset \R^N$ is compact, and $D$ is dense in $\R^N$, there exist $x_1,\dots x_\nu\in D$ such that $ B_j \subset \bigcup_{i=1}^\nu B_\delta(x_i)$, where $B_\delta(x)$ denotes the closed ball of $\R^N$ of radius $\delta$ centered at $x\in\R^N$.
For $i=1,\dots,\nu$, let us consider continuous functions $\phi_i:\R^N\to[0,1]$  such that
\begin{equation*}
\supp(\phi_i)\subset B_\delta(x_i)\, \qquad \mathrm{and}\qquad \sum_{i=1}^\nu \phi_i(x)=1 \quad \forall\, x\in B_j \, ,
\end{equation*}
and define the functions
\begin{equation}\label{ecu:estrellas}
f^*_n(t,x)=\sum_{i=1}^\nu\phi_i(x)\, f_n(t,x_i)\qquad \mathrm{and}\qquad
f^*(t,x)=\sum_{i=1}^\nu \phi_i(x)\,  f(t,x_i)\, .
\end{equation}
Then, for any $x(\cdot)\in C(I, B_j )$ we have that
\begin{multline}\label{eq:25.05-19:47}
\normp{  f_{n}\big(\cdot,x(\cdot)\big)  -   f\big(\cdot,x(\cdot)\big)}{I}\le \normp{  f_{n}\big(\cdot,x(\cdot)\big)  -   f^*_{n}\big(\cdot,x(\cdot)\big)}{I}\\
			+\normp{   f^*_{n}\big(\cdot,x(\cdot)\big) -   f^*\big(\cdot,x(\cdot)\big)}{I}+\normp{  f^*\big(\cdot,x(\cdot)\big) -   f\big(\cdot,x(\cdot)\big)}{I}\, .
\end{multline}
Let us separately analyze each element in the sum on the right-hand side of equation~\eqref{eq:25.05-19:47}. As regard the first one, we have that
\begin{equation}\label{eq:27.10-13:59}
\begin{split}
\normp{  f_{n}\big(\cdot,x(\cdot)\big)  -   f^*_{n}\big(\cdot,x(\cdot)\big)}{I}^p&=\int_I\Big|\sum_{i=1}^\nu \phi_i\big(x(t)\big)\left[ f_{n}\big(t,x(t)\big)- f_{n}(t,x_i)\right]\Big|^pdt\\
&\le \int_I\Big( \sum_{i=1}^\nu \phi_i\big(x(t)\big)\left| f_{n}\big(t,x(t)\big)- f_{n}(t,x_i)\right|\Big)^pdt\\
&\le \int_I\Big(\sum_{i=1}^\nu \phi_i\big(x(t)\big)\, l^j_{n}(t) \left| x(t)-x_i\right|\Big)^pdt\\
&\le \int_I \Big( \sum_{i=1}^\nu \phi_i\big(x(t)\big)\, l^j_{n}(t)\ \delta\Big)^pdt\\
&= \frac{1}{\rho}\left(\frac{\ep}{3}\right)^p\int_I\left(l^j_{n}(t)\right)^p\, dt\le \left(\frac{\ep}{3}\right)^p\, .
\end{split}
\end{equation}
As regard  the third element of the sum in \eqref{eq:25.05-19:47}, recall that, due to the Propositions~\ref{prop:25.05-10:53} and~\ref{prop:07.07-19:44}, the $l$-bound $\bar{l}^j(\cdot)\in L^p_{loc}$  on $B_j$ for $f$ satisfies
\begin{equation*}
\int_I\left(\bar{l}^j(s)\right)^p\,ds<\rho\,.
\end{equation*}
Therefore, reasoning like in \eqref{eq:27.10-13:59}, we obtain that
\begin{equation}\label{eq:27.10-14:00}
\normp{  f^*\big(\cdot,x(\cdot)\big) -   f\big(\cdot,x(\cdot)\big)}{I}\le \frac{\ep}{3}\, ,
\end{equation}
and notice that both, \eqref{eq:27.10-13:59} and \eqref{eq:27.10-14:00}, are independent of $x(\cdot)\in C(I, B_j )$. \par
Finally, since $(f_{n})_\nin$ converges to $f$ in $\left(\LC\big(\R^M\big),\T_{D}\right)$, consider $n$ big enough so that $\normp{f_{n}(\cdot,x_i) -   f(\cdot,x_i)}{I}<\ep/(3\,\nu)$ for any $i=1,\dots, \nu$. Then, from the expressions~\eqref{ecu:estrellas} and the fact that $\phi_i(x)\le 1$ for each $x\in\R^N$ we deduce that
\begin{equation}\label{eq:27.10-16:29}
\normp{   f^*_{n}\big(\cdot,x(\cdot)\big) -   f^*\big(\cdot,x(\cdot)\big)}{I}\le \sum_{i=1}^\nu \normp{   f_{n}(\cdot,x_i) -   f(\cdot,x_i)}{I}\le \frac{\ep}{3}\,.
\end{equation}
Gathering together \eqref{eq:27.10-13:59}, \eqref{eq:27.10-14:00} and \eqref{eq:27.10-16:29}, we obtain the result.
\end{proof}
When dealing with a function in $\LC\big(\R^M\big)$ with $L^p_{loc}$-bounded $l$-bounds, the previous theorem provides, as a corollary, a condition of equivalence of the hulls.
\begin{cor}
Let $f$ be a function in $\LC\big(\R^M\big)$ with $L^p_{loc}$-bounded $l$-bounds, then
\begin{equation*}
\mathrm{Hull}_{(\SC(\R^M),\T_1)}(f)=\mathrm{Hull}_{(\SC(\R^M),\T_2)}(f)\, ,
\end{equation*}
where $\T_1$ and $\T_2$ are any of the previously introduced topologies.
\label{cor:HullT1=HullT2}
\end{cor}
A characterization of compactness in $L^p_{loc}\big(\R^M\big)$, when $1\le p<\infty$, has been given in \cite{paper:GS1}, where it is proved that $E\subset L^p_{loc}\big(\R^M\big)$ is relatively compact if and only if the following conditions hold:
\begin{itemize}
\item[(i)]  for every compact interval $I\subset \R$ there exists a constant $c=c(I)$ such that $\int_I|f(t)|^pdt\le c$, for every $f\in E$, and
\item[(ii)]  for every $\ep>0$ and  for every compact interval $I\subset \R$ there exists a $\delta=\delta(\ep,I)>0$ such that, if $|\tau|\le\delta$, then $\int_I|f(t+\tau)-f(t)|^pdt\le \ep$, for every $f\in E$.
\end{itemize}
Moreover, a sufficient condition for the relative compactness of a set $E\subset\LC\big(\R^M\big)$ in $(\LC,T_B)$ is also given in the same reference. Next we characterize such a compactness under the assumption that the set $E$ admits $L^p_{loc}$-bounded $l$-bounds.
\begin{thm}\label{thm:compactness_in_LC}
Let  $E\subset \LC\big(\R^M\big)$  admit $L^p_{loc}$-bounded $l$-bounds, $\T$ be any of the previously introduced topologies, and $D$ be a countable dense subset of $\R^N$. The following statements are equivalent.
\begin{itemize}
\item[(i)] The space $(E,\T)$ is relatively compact.
\item[(ii)] For any fixed $x\in D$ the set $\left\{f_x(\cdot)=f(\cdot,x)\mid f\in E\right\}$ is  relatively compact in $L^p_{loc}\big(\R^M\big)$.
\end{itemize}
\end{thm}
\begin{proof}
Firstly, recall that, since $E$  has $L^p_{loc}$-bounded $l$-bounds,  all the considered topologies are equivalent thanks to Theorem~\ref{thm:22.07-11:24}, and thus we will work with $(E,\T_D)$. (i) $\Rightarrow$ (ii) is straightforward. \par\vspace{0.05cm}
(ii) $\Rightarrow$ (i). Consider a sequence $(f_n)_\nin$ in $ E$ , fix $j\in\N$ and, for any $\nin$, let $l^j_n(\cdot)$ be the optimal $l$-bound for $f_n$ on $B_j$. Moreover, let $D_j$ be the set  $D\cap B_j$. By hypothesis, for any $x\in D_j$ the set $\{f_n(\cdot, x)\mid \nin\}$ is relatively compact in  $L^p_{loc}$; therefore, using a diagonal argument, we obtain a subsequence of $(f_n)_\nin$, which we keep denoting with the same indexes, such that
\begin{equation*}
f_n(t,x)\xrightarrow{\nti}f(t,x)\quad \text{for a.e. } t\in\R,\ \forall \ x\in D_j \,.
\end{equation*}
Moreover, by hypothesis the set $\{l^j_n(\cdot)\mid \nin\}$ is weakly bounded in $L^p_{loc}$ and thus, reasoning like in the proof of  Proposition~\ref{prop:25.05-10:53}, we obtain a function $l^j(\cdot)\in L^p_{loc}$ such that for any $x,y\in D_j$ the following inequality holds
\begin{equation*}
|f(t,x)-f(t,y)|\le l^j(t)\,|x-y| \quad \text{for a.e. } t\in\R \,.
\end{equation*}
A continuous extension of $f$ to the entire ball $B_j$ is given by
\begin{equation*}
f(t,x)=\lim_{n\to\infty}f(t,x_n)\quad \text{whenever } x\in B_j, \, (x_n)_{n\in\N}\text{ in } D_j, \text{ and } x_n\to x\, .
\end{equation*}
The definition is well-posed; indeed if  $(x_n)_{n\in\N}$ and $(y_n)_{n\in\N}$ are sequences in $D_j$ such that  $x_n\to x$ and $y_n\to x$, then
\begin{equation*}
|f(t,x_n)-f(t,y_n)|\le l^j(t)|x_n-y_n| \quad \text{for a.e. } t\in\R \,,
\end{equation*}
and the right-hand side goes to zero as $n\to\infty$, for almost every $t\in\R$. It is straightforward to prove that $l^j$ keeps being a Lipschitz coefficient for the $f$ defined on the whole ball $B_j$. Finally, standard arguments of measure theory allow to prove that $f$ also satisfies properties (C1) and (C2) and therefore $f\in\LC$.
\end{proof}
As a corollary of Theorem \ref{thm:compactness_in_LC} and of the conditions (i) and (ii) listed before it, we obtain a characterization of the compactness of $\mathrm{Hull}_{(\LC(\R^M),\T)}(f)$ when $f\in\LC$ admits $L^p_{loc}$-bounded $l$-bounds.
\begin{cor}\label{cor:compactnessHULL(f)}
Let  $f\in \LC\big(\R^M\big)$  admit $L^p_{loc}$-bounded $l$-bounds, $\T$ be any of the previously introduced topologies, and $D$ be a countable dense subset of $\R^N$. The following statements are equivalent.
\begin{itemize}
\item[(i)] $\mathrm{Hull}_{(\LC(\R^M),\T)}(f)$ is compact.
\item[(ii)] For every $x\in D$ the map $\R\to L^p_{loc}\big(\R^M\big)$, $t\mapsto f_t(\cdot,x)$ is bounded and uniformly continuous.
\end{itemize}
\end{cor}
\section{Continuity with respect to the initial conditions for ODEs}\label{secContFlowODEs}
This section deals with the continuity of the solutions with respect to the variation of the initial data and of the coefficients. All the proofs will be given for $p=1$ remembering that, if $I\subset\R$ is a bounded interval, then, $L^p(I)\subset L^1(I)$. For the sake of completeness and to set some notation, we state a theorem of existence and uniqueness of the solution for a Cauchy Problem of Carath\'eodory type. A proof can be found in Coddington and Levinson \cite[Theorems 1.1, 1.2 and 2.1]{book:CL}.
\begin{thm}
For any $f\in\LC$ and any $x_0\in\R^N$ there exists a maximal interval $I_{f,x_0}=(a_{f,x_0},b_{f,x_0})$ and a unique continuous function $x(\cdot,f,x_0)$ defined on $I_{f,x_0}$ which is the solution of the Cauchy Problem
\begin{equation}\label{eq:solLCE}
\dot x=f(t,x)\, ,\qquad x(0)=x_0\, .
\end{equation}
In particular, if $a_{f,x_0}>-\infty$ (resp. $b_{f,x_0}<\infty$), then $|x(t,f,x_0)|\to\infty$ as $t\to a_{f,x_0}$ (resp. as $t\to b_{f,x_0}$).
\label{thm:05.07-13:44}
\end{thm}
\begin{cor}\label{cor:14:07-21:05}
Let $\Theta$ be a suitable set of moduli of continuity. For any $f\in\LC$, $F\in\TC\big(\R^{N\times N}\big),\, h\in\TC$, and $x_0,\, y_0\in\R^N$, there exists a unique solution of the Cauchy problem
\begin{equation}
\begin{cases}
\dot x=f(t,x), &x(0)=x_0\,,\\
\dot y=F(t,x)\,y+h(t,x), &y(0)= y_0\,,
\end{cases}
\label{eq:system}
\end{equation}
which will be denoted by $(x(\cdot,f,x_0),y(\cdot,f,F,h,x_0, y_0))$, and whose maximal interval of definition coincides with the interval $I_{f,x_0}$ provided by {\rm Theorem~\ref{thm:05.07-13:44}}.
\end{cor}
\begin{defn}\label{def:THETAjseq}
Let $E\subset\LC$ with $L^1_{loc}$-equicontinuous $m$-bounds. For any $j\in\N$ and for any interval $I=[q_1,q_2]$, $q_1,q_2\in\Q$, define
\begin{equation*}
\theta^I_j(s):=
 \sup_{t\in I,f\in E}\int_t^{t+s}m_f^j(u)\, du\, .
\end{equation*}
where, for any $f\in E$, the function $m_f^j(\cdot)\in L^p_{loc}$ denotes the optimal $m$-bounds of $f$ on $B_j$. Notice that, since $E$ admits $L^1_{loc}$-equicontinuous $m$-bounds, then  $\Theta=\{\theta^I_j(\cdot)\mid I=[q_1,q_2],\, q_1,q_2\in\Q,\, j\in\N\}$ defines a suitable set of moduli of continuity.
\end{defn}
\begin{rmk}\label{rmk:THETAjfunc}
If $f\in\LC$ has $L^1_{loc}$-equicontinuous $m$-bounds we similarly define for any $B_j\subset\R^N$,
\begin{equation*}
\theta_j(s):= \sup_{t\in\R}\int_t^{t+s}m^j(u)\, du\, ,
\end{equation*}
where $m^j(\cdot)$ is the optimal $m$-bound for $f$ on $B_j$. Here again, notice that $\Theta=\{\theta^I_j(\cdot)\mid I=[q_1,q_2],\, q_1,q_2\in\Q,\, j\in\N\}$  defines a suitable set of moduli of continuity thanks to the $L^1_{loc}$-equicontinuity.
\end{rmk}
Now we prove several theorems of continuity assuming the existence of $L^1_{loc}$-equicontinuous $m$-bounds.
\begin{thm}\label{thm:ContODETthetaE}
Consider $E\subset\LC$  with $L^1_{loc}$-equicontinuous $m$-bounds and let $\Theta=\{\theta^I_j\mid I=[q_1,q_2],\,  q_1,q_2\in\Q,\,  j\in\N\}$ be the countable family of moduli of continuity in {\rm Definition~\ref{def:THETAjseq}}. With the notation of {\rm Theorem~\ref{thm:05.07-13:44}} and {\rm Corollary~\ref{cor:14:07-21:05}},
\begin{itemize}
\item[(i)] if $(f_n)_\nin$  in $E$ converges to $f$ in $(\LC,\T_\Theta)$ and $(x_{0,n})_\nin$ in $\R^N$ converges to $x_0\in\R^N$, then
    \[ \qquad x(\cdot,f_n,x_{0,n}) \xrightarrow{\nti}  x(\cdot,f,x_0)\]
     uniformly in any $[T_1,T_2]\subset I_{f,x_0}$;\par\vspace{0.05cm}
\item[(ii)] moreover, if $(F_n)_\nin$ in $\TC\big(\R^{N\times N}\big)$ converges to $F$ in $\big(\TC\big(\R^{N\times N}\big),\T_\Theta\big)$, $(h_n)_\nin$ in $\TC$ converges to $h$ in $(\TC,\T_\Theta)$, and  $( y_{0,n})_\nin$ in $\R^N$ converges to $ y_0\in\R^N$, then
    \[  \hspace{1cm} y(\cdot,f_n,F_n,h_n,x_{0,n}, y_{0,n})\xrightarrow{\nti} y(\cdot,f,F,h,x_0,  y_0)\]
uniformly  in any $[T_1,T_2]\subset I_{f,x_0}$.
\end{itemize}
\end{thm}
\begin{proof}
(i) We will prove the uniform convergence of $\big(x(\cdot,f_n,x_{0,n})\big)_\nin$ to  $x(\cdot,f,x_0)$ in $[0,T]$ for any $0< T<b_{f,x_0}$. The case $a_{f,x_0}<T<0$ is analogous. Denote
\begin{equation}
0<\rho =1+ \max\big\{ (|x_{0,n}|)_\nin,\  \|x(\cdot,f,x_0)\|_\infty\big\}\, ,
\label{eq:DefRHO}
\end{equation}
and define
\begin{equation*}
z_n(t)=
\begin{cases}
x(t,f_n,x_{0,n}), & \text{if $0\le t< T_n$,} \\
x(T_n,f_n,x_{0,n}), & \text{if $T_n\le t\le T$.}
\end{cases}
\end{equation*}
where $T_n=\sup\{t\in[0,T]\mid|x(s,f_n,x_{0,n})|\le \rho,\, \forall\, s\in[0,t]\}$. Notice that by \eqref{eq:DefRHO} and by the continuity of $\big(x(\cdot,f_n,x_{0,n})\big)_\nin$, we have that $T_n>0$ for any $n\in\N$. In particular notice that $(z_n(\cdot))_{\nin}$ is uniformly bounded. Moreover, consider $j\in\N$ so that $\rho<j$ and let $(m_n(\cdot))_{\nin}=(m_{f_n}^j(\cdot))_{\nin}$ be the sequence of optimal $m$-bounds of $(f_n)_{\nin}$ on $B_j$. Now, if $t_1,t_2\in[0,T_n)$, $t_1<t_2$, then
\begin{equation}
|z_n(t_1)-z_n(t_2)|\le \int_{t_1}^{t_2}\big|f_n\big(s,z_n(s)\big)\big|\, ds \le \int_{t_1}^{t_2}m_n(s)\, ds\, .
\label{eq:14.06-18:11}
\end{equation}
Fixed $\ep>0$,  since $E$ admits $L^1_{loc}$-equicontinuous $m$-bounds, there exists $\delta=\delta(T,\ep)>0$ such that, if $0\le t_1 \le t_2< T_n$, then the right-hand side in \eqref{eq:14.06-18:11} is smaller than $\ep$ whenever $t_2-t_1<\delta$. Notice that in facts the inequality $|z_n(t_1)-z_n(t_2)|<\ep$ is true on the whole interval $[0,T]$ whenever $t_2-t_1<\delta$ because in $[T_n,T]$ the difference on the left side of equation \eqref{eq:14.06-18:11} is zero. Thus, the sequence  $(z_n(\cdot))_{\nin}$ is equicontinuous. Then, Ascoli-Arzel\'a's theorem implies that $(z_n(\cdot))_{\nin}$ converges uniformly, up to a subsequence, to some continuous function $z:[0,T]\to \R^N$.
\par\vspace{0.05cm}
In order to conclude the proof, we prove that $z(\cdot)\equiv x(\cdot,f,x_0)$ in $[0,T]$. Define
\begin{equation}
T_0=\sup \{t\in[0,T]\mid |z(s)|<\rho-1/2\quad \forall \, s\in[0,t]\}\, ,
\label{eq:defT0}
\end{equation}
and notice that $T_0>0$ because  $(x_{0,n})_\nin$ converges to $x_0$ and $z(\cdot)$ is continuous.  Since $z_n(\cdot)$ converges uniformly to $z(\cdot)$ in $[0,T]$, then there exists $n_0\in\N$ such that if $n>n_0$, then
\begin{equation*}
|z_n(t)|<\rho-1/4\qquad \forall \, t\in[0,T_0]\, .
\end{equation*}
Therefore, for any $t\in[0,T_0]$  and for any $n>n_0$ one has  $z_n(t)=x(t,f_n,x_{0,n})$ and thus
\begin{equation}
z_n(t)=x_{0,n}+\int_0^tf_n\big(s,z_n(s)\big)\, ds\, ,\qquad t\in[0,T_0]\, ,\ n>n_0\, .
\label{eq:14.06-19:35}
\end{equation}
Now consider the compact set $\mathcal{K}=\{z_n(\cdot)\mid\nin\}\cup\{z(\cdot)\}\subset C\big([0,T],\R^N\big)$ and notice that $\mathcal{K}\subset\mathcal{K}^I_j$ for some $I=[q_1,q_2], \ q_1,q_2\in\Q$ and some $j\in\N$, up to an extension by constants of the functions in $\mathcal{K}$ to the whole interval $I$. Moreover, recall that $(f_n)_{\nin}$ converges to $f$ in $\T_\Theta$, $(z_n(\cdot))_{\nin}$ converges uniformly to $z(\cdot)$ in $[0,T]$ and $(x_{0,n})_{\nin}$ converges to $x_0$ as $\nti$. Then, passing to the limit in \eqref{eq:14.06-19:35}, we have that
 \begin{equation*}
z(t)=x_0+\int_0^tf\big(s,z(s)\big)\, ds\, ,\qquad t\in[0,T_0]\, .
\end{equation*}
In other words $z(t)$ is actually the solution of \eqref{eq:solLCE} in $[0,T_0]$. Additionally, it is easy to check that $|z(t)|\le\rho-1$. We prove that $T_0=T$ in order to conclude the proof. Otherwise, by \eqref{eq:defT0} and by the continuity of $z(\cdot)$, one would have $|z(T_0)|=|x(T_0,f,x_0)|=\rho-1/2$, which contradicts \eqref{eq:DefRHO}. Hence, $T_0=T$, as claimed, and thus for any  $t\in [0,T]$ we have that $x(t,f,x_0)= z(t)$ and $x(t,f_n,x_{0,n})= z_n(t)$ for any $\nin$, which concludes the proof of (i).\par\vspace{0.05cm}
(ii) The continuous dependence in the first component is given by part (i). In order to simplify the notation, let us denote by $x_n(\cdot)=x(\cdot,f_n,x_{0,n})$,  $y_n(\cdot)=y(\cdot,f_n,F_n,h_n,x_{0,n}, y_{0,n})$, $x(\cdot)=x(\cdot,f,x_0)$, and $y(\cdot)=y(\cdot,f,F,h,x_0,  y_0)$. Moreover, call $\widetilde F_n(t)=F_n(t,x_n(t))$, $\widetilde F(t)=F(t,x(t))$, $\widetilde h_n(t)=h_n(t,x_n(t))$ and $\widetilde h(t)=h(t,x(t))$. If we  prove that $(\widetilde F_n(\cdot))_\nin$ and $\widetilde h_n(\cdot))_\nin$ converge in $L^p_{loc}$ to $\widetilde F(\cdot)$ and $\widetilde h(\cdot)$ respectively, then we have the thesis applying Lemma IV.9 in \cite{book:GS} to the linear case. Therefore, let us fix an interval $I\subset \R$. Then,
\begin{equation*}
\begin{split}
\big\|\widetilde F_n(\cdot)-\widetilde F(\cdot)\big\|_p&=\normp{ F_n\big(\cdot,x_n(\cdot)\big)-F\big(\cdot, x(\cdot)\big)}{I}\\
&\le \normp{ F_n\big(\cdot,x_n(\cdot)\big)-F\big(\cdot, x_n(\cdot)\big)}{I}+\normp{ F\big(\cdot,x_n(\cdot)\big)-F\big(\cdot, x(\cdot)\big)}{I}\\
&\le \sup_{\xi\in\K^I_j}\normp{ F_n\big(\cdot,\xi(\cdot)\big)-F\big(\cdot, \xi(\cdot)\big)}{I}+\normp{ F\big(\cdot,x_n(\cdot)\big)-F\big(\cdot, x(\cdot)\big)}{I}\,
\end{split}
\end{equation*}
where $j\in\N$ is chosen as in part (i). When $\nti$, the right-hand side of the previous inequality goes to zero because $(F_n)_\nin$ converges to $F$ in $\TC\big(\R^{N\times N}\big)$ and $F\in\TC\big(\R^{N\times N}\big)$. Analogous reasonings apply to the sequence $(\widetilde h_n(\cdot))_\nin$.  Therefore, we have  the required $L^p_{loc}$ convergences and thus uniform convergence of the solutions of the nonhomogeneus linear equation.
\end{proof}
Let $f\in\LC$, $\Theta=(\theta_j)_{j\in\N}$ be a suitable set of moduli of continuity and consider the family of differential equations $\dot x =g(t,x)$, where $g\in\mathrm{Hull}_{(\LC,\T_\Theta)}(f)$. With the notation introduced in Theorem \ref{thm:05.07-13:44}, let us denote by $\U_1$ the subset of $\R\times\mathrm{Hull}_{(\LC,\T_\Theta)}(f)\times\R^N$ given by
\begin{equation*}
\U_1=\bigcup_{\substack{g\in\mathrm{Hull}_{(\LC,\T_\Theta)}(f)\,,\\x\in\R^N}}
\{(t,g,x)\mid t\in I_{g,x}\}\,.
\end{equation*}
Let $f\in\LC$, $F\in \TC(\R^{N\times N})$, $h\in \TC$ and consider the family of differential equations of type~\eqref{eq:system} for $(g,G,k)\in \Hu=\mathrm{Hull}_{\left(\LC\times \TC\times\TC,\T_\Theta\times\T_\Theta\times\T_\Theta\right)}(f,F,h)$, where the hull is constructed as in Definition~\ref{def:Hull}. Denote by $\U_2$ the subset of $\R\times\Hu\times\R^N\!\!\times\R^N$ given by
\begin{equation*}
\U_2=\bigcup_{\substack{(g,G,k)\in\Hu\,,\\ x_0\in\R^N}}\{(t,g,G,k,x_0,y_0)\mid t\in I_{g,x_0}\,,\; y_0\in\R^n\}\,.
\end{equation*}
With the previous notation we can state the following theorem.
\begin{thm}
Let $f\in\LC$ have $L^1_{loc}$-equicontinuous $m$-bounds and let $\Theta=(\theta_j)_{j\in\N}$ be the sequence of functions defined in {\rm Remark \ref{rmk:THETAjfunc}}.
\begin{itemize}
\item[(i)] The set $\U_1$ is open in $\mathrm{Hull}_{(\LC,\T_\Theta)}(f)\times\R^N$ and the map
\begin{equation*}
\begin{split}
\Phi_1\colon  \ \U_1\subset \R\times\mathrm{Hull}_{(\LC,\T_\Theta)}(f)\times\R^N\  &\to\ \ \mathrm{Hull}_{(\LC,\T_\Theta)}(f)\times\R^N\\
\ (t,g,x_0)\quad \qquad \qquad &\mapsto\qquad \big(g_t, x(t,g,x_0)\big)\,,\\
\end{split}
\end{equation*}
defines a local continuous skew-product flow on $\mathrm{Hull}_{(\LC,\T_\Theta)}(f)\times\R^N$.
\item[(ii)] The set $\U_2$ is open in $\R\times \Hu\times \R^N\times
\R^N$ and the map
\begin{equation*}
\begin{split}
\qquad\quad\;\;\Phi_2\colon \ \U_2\subset \R\times\Hu\times \R^N\!\!\times \R^N &\to\qquad\qquad\qquad\Hu\times \R^N \!\!\times \R^N\\
 (t,g,G,k,x_0, y_0) \  \ & \mapsto\;  \big(g_t, G_t, k_t ,x(t,g,x_0),y(t,g,G,k,x_0, y_0)\big)\,,\\
\end{split}
\end{equation*}
defines a local continuous skew-product flow on $\Hu\times \R^N\!\!\times \R^N$.
\end{itemize}
\label{thm:contFlowTtheta}
\end{thm}
\begin{proof}
The proof is a direct consequence of Theorem \ref{thm:ContODETthetaE} and Theorem \ref{prop:25.05-18:53}.
\end{proof}
We conclude this part giving a theorem of existence of the solutions for differential problems whose vector fields are in $\TC$, i.e. not necessarily continuos in the space variables either. The underlying condition is that such vector fields are limit of sequences in $\SC$ with $L^1_{loc}$-equicontinuous $m$-bounds in the topology $\T_\Theta$, where $\Theta$ is the suitable set of moduli of continuity given in Definition~\ref{def:THETAjseq}.
\begin{thm}
Let $(f_n)_\nin$ be a sequence in $\SC$ with $L^1_{loc}$-equicontinuous $m$-bounds and $\Theta$ be the suitable set of moduli of continuity given in  {\rm Definition~\ref{def:THETAjseq}}. Assume that $(f_n)_\nin$ converges to some $f$ in $(\TC,\T_\Theta)$ and that $(x_{0,n})_\nin$ is a sequence in $\R^N$ converging to $x_0\in\R^N$. Then, denoting by $x_n(\cdot)$ a solution of the differential problem $\dot x= f_n(t,x)$ defined in the maximal interval $(a_n,b_n)$ and such that $0\in (a_n,b_n)$ and $x_n(0)=x_{0,n}$, we have
\begin{itemize}
\item[(i)] $\limsup_\nti a_n=a^*<0$, and $\liminf_\nti b_n=b^*>0$.
\item[(ii)] There exist $a^*<a<b<b^*$ and a continuous function $x(\cdot)$ such that, up to a subsequence,
\begin{equation*}
x_n(\cdot) \xrightarrow{\nti}  x(\cdot)
\end{equation*}
uniformly on the compact subsets of $(a,b)$.
\item[(iii)] For every $s,t\in(a,b)$, the function $x(\cdot)$ satisfies
\begin{equation*}
x(t)=x(s)+\int_s^tf\big(u,x(u)\big)\, du\, .
\end{equation*}
\end{itemize}
\end{thm}
\begin{proof}
We prove the existence of $x(\cdot)$ in $[0,b)$. The other case is analogous. Consider the constant
\begin{equation*}
0<\rho =1+ \max\big\{ |x_{0,n}|\mid\nin\big\}\, ,
\end{equation*}
and for every $\nin$ define $z_n\colon [0,\infty)\to\R^N$ by
\begin{equation*}
z_n(t)=
\begin{cases}
x_n(t), & \text{if $0\le t\le T_n$,} \\
x_n(T_n), & \text{if $T_n<t< \infty$.}
\end{cases}
\end{equation*}
where $T_n=\sup\{t\ge0\mid|x_n(s)|\le \rho\ \,\forall\, s\in[0,t]\}$. Then, the same arguments used in the proof of Theorem~\ref{thm:ContODETthetaE} provide a $T_0>0$, a subsequence of $\big(x_n(\cdot)\big)_\nin$, that we keep denoting with the same indexes, and a continuous function $x\colon [0,T_0]\to\R^N$ such that
\begin{equation*}
x_n(\cdot) \xrightarrow{\nti}  x(\cdot)\,,
\end{equation*}
uniformly in $[0,T_0]$, $|x(t)|<\rho-1/2$ for every $t\in[0,T_0]$, and moreover
\begin{equation}
x(t)=x(0)+\int_0^tf\big(s,x(s)\big)\, ds\, ,\qquad \text{for every $t\in[0,T_0]$}\,,
\label{eq:07.10_13:34}
\end{equation}
which means that $x(\cdot)$ is an absolutely continuous function solving the Carath\'eodory differential equation $\dot x=f(t,x)$, where $f\in\TC$. Notice that such a problem and the integral solution of \eqref{eq:07.10_13:34} are well-defined thanks to Proposition \ref{prop:coincideAE=>CoincideOnTheContinuous}. Finally, standard arguments of Carath\'eodory ODEs allow to extend the function $x(\cdot)$ to the maximal interval $(a,b)$.
\end{proof}
In the second part of the section, we will give several theorems of continuity of the solutions of Carath\'eodory differential systems whose vector fields are in a set $E\subset\LC$ admitting $L^p_{loc}$-bounded $l$-bounds. In particular, if $C\subset E$ is compact with respect to any of the considered topology, then the solutions of the differential equations of type \eqref{eq:solLCE}, where the vector fields belong to $C$, determine a suitable set of moduli of continuity.
\begin{thm}
Consider $E\subset\LC$  with $L^p_{loc}$-bounded $l$-bounds.
\begin{itemize}
\item[(i)] If $(f_n)_\nin$ in $E$ converges to $f$ in $(\LC,\T_D)$ and $(x_{0,n})_\nin$ in $\R^N$ converges to $x_0\in\R^N$, then
    \[ \qquad x(\cdot,f_n,x_{0,n}) \xrightarrow{\nti}  x(\cdot,f,x_0)\]
     uniformly in any $[T_1,T_2]\subset I_{f,x_0}$.
\item[(ii)] Let $C\subset E$ be compact with respect to $\T_D$ and, for any interval $I=[q_1,q_2]\subset\R$, with $q_1,q_2\in\Q$, and any $j\in\N$, define
\begin{equation}
C^I_j=\left\{x\colon J\to B_j\, \middle|\,  \begin{split}
J\subset I &\text{ interval, and } \\[-0.1cm]
 \exists \,f\in C&\ \text{such that}\ \forall \, s,t\in J\\[-0.15cm]
 x(t)= x(s)&+{\textstyle\int_s^tf(u,x(u))\, du }
\end{split} \right\}\,.
\label{eq:07.10_19:49}
\end{equation}
Then, each of the sets $C^I_j$ is equicontinuous and, denoted by $\theta^I_j$ its modulus of continuity, the set
\begin{equation*}
\hspace{1.3cm}\Theta=\left\{\theta^I_j\in C(\R^+,\R^+)\, \middle| \, \begin{split}
I=[q_1&,q_2]\subset\R,\text{ with } q_1,q_2\in\Q, \ j\in\N,\, \\
\theta^I_j &\text{ modulus of continuity of } C^I_j\\
\end{split}\right\}\,,
\end{equation*}
 is a suitable set of moduli of continuity.
\item[(iii)] Let $C\subset E$ be compact with respect to $\T_D$ and $\Theta$ be the suitable set of moduli of continuity given by {\rm (ii)}. If $(f_n)_\nin$ in $ C$ converges to $f$ in $(\LC,\T_D)$, $(F_n)_\nin$ in $ \TC\big(\R^{N\times N}\big)$ converges to $F$ in $\big(\TC\big(\R^{N\times N}\big),\T_\Theta\big)$, $(h_n)_\nin$ in $ \TC$ converges to $h$ in $(\TC,\T_\Theta)$, and  $(x_{0,n},  y_{0,n})_\nin$ in $\R^N\!\!\times\R^N$ converges to $(x_0, y_0)\in\R^N\!\!\times\R^N$, then
    \[  \qquad\quad \big(x(\cdot,f_n,x_{0,n}),y(\cdot,f_n,F_n,h_n,x_{0,n}, y_{0,n})\big)\to \big(x(\cdot,f,x_0),y(\cdot,f,F,h,x_0,  y_0)\big)\]
    as $n\to\infty$, uniformly  in any $[T_1,T_2]\subset I_{f,x_0}$.
\end{itemize}
\label{thm:ContODETpE}
\end{thm}
\begin{proof}
(i) Since $E$ has $L^p_{loc}$-bounded $l$-bounds, by Theorem~\ref{thm:22.07-11:24} the convergence in $(\LC,\T_D)$ implies the convergence in $(\LC,\T_B)$. The proof closely follows the one given in Theorem~\ref{thm:ContODETthetaE}, with the exception that, instead of \eqref{eq:14.06-18:11}, now we have
\begin{equation}
\begin{split}
|z_n(t_1)-z_n(t_2)|&\le \int_{t_1}^{t_2}\big|f_n\big(s,z_n(s)\big)\big|\, ds\\
&\le  \int_{t_1}^{t_2}\big|f_n\big(s,z_n(s)\big)-f\big(s,z_n(s)\big)\big|\, ds+\int_{t_1}^{t_2}\big|f\big(s,z_n(s)\big)\big|\, ds\\
& \le \int_{t_1}^{t_2}\big|f_n\big(s,z_n(s)\big)-f\big(s,z_n(s)\big)\big|\, ds+ \int_{t_1}^{t_2}m_{f}^j(s)\, ds\, ,
\end{split}
\label{eq:27.06-21:16}
\end{equation}
where $m_{f}^j(\cdot)\in L^p_{loc}$ is the optimal $m$-bound for $f$ on $B_j$ and notation of Theorem~\ref{thm:ContODETthetaE} is used. Fixed $\ep>0$, due to the convergence of $(f_n)_\nin$ to  $f$ in $(\LC,\T_B)$, there exists an $n_0\in\N$ such that, if $n>n_0$, then
\begin{equation}
\sup_{{k\in\N}}\int_{t_1}^{t_2}\big|f_n\big(s,z_k(s)\big)-f\big(s,z_k(s)\big)\big|\, ds<\ep\, .
\label{eq:27.06-21:18}
\end{equation}
Notice that $\{z_k\mid k\in\N\}$ is a bounded set of continuous functions. On the other side, by the absolute continuity of the integral, there exists $\delta>0$ such that if $0<t_2-t_1<\delta$, then
\begin{equation}
 \int_{t_1}^{t_2}\big|f_n\big(s,z_n(s)\big)-f\big(s,z_n(s)\big)\big|\, ds<\ep \qquad \forall n=1,\dots,n_0\,,
\label{eq:27.06-21:25}
\end{equation}
and also
\begin{equation}
\int_{t_1}^{t_2}m_{f}^j(s)\, ds<\ep\,.
\label{eq:27.06-21:29}
\end{equation}
Gathering the inequalities \eqref{eq:27.06-21:16}, \eqref{eq:27.06-21:18}, \eqref{eq:27.06-21:25} and \eqref{eq:27.06-21:29}, we obtain a common modulus of continuity for all the functions in $\{z_i\mid i\in\N\}\cup\{z\}$. The rest of the proof follows the arguments of Theorem~\ref{thm:ContODETthetaE}.\par\vspace{0.05cm}
(ii) Let us fix $\ep>0$, $I=[q_1,q_2]\subset\R$, with $q_1,q_2\in\Q$, and $j\in\N$, and consider the following seminorm defined on $E$
\begin{equation*}
p_j(f)=\sup_{z\in C(I,B_j)}\int_I\big|f\big(t,z(t)\big)\big|\, dt,\qquad f\in E\,.
\end{equation*}
Moreover, for any $\widetilde f\in C$, denote by $U_{\ep/2}^j(\widetilde f)$ the following set
\begin{equation*}
U_{\ep/2}^j(\widetilde f) =\big\{f\in E\ \mid \ p_j\big(f-\widetilde f\, \big)\le \ep/2\big\}\,.
\end{equation*}
Therefore, by the compactness of $C$, there exist $\nu\in\N$ and $f_1,\dots, f_{\nu}\in C$ such that
\begin{equation*}
C\subset\bigcup_{i=1}^{\nu} U_{\ep/2}^j(f_i)\,.
\end{equation*}
For any $i=1,\dots,\nu$, denote by $m_i(\cdot)$ the $m$-bound of $f_i$ on $B_j$ and notice that there exists $\delta>0$ such that, if $s,t\in I$ and $|t-s|<\delta$, then
\begin{equation*}
\int_s^tm_i(u)\, du\le \ep/2\,, \quad \forall\, i=1,\dots,\nu\,.
\end{equation*}
Now, consider $x\colon J\to B_j$, with $x(\cdot)\in C^I_j$ and possibly extend it by constants to the whole interval $I$. Also, by the definition of $C^I_j$ in \eqref{eq:07.10_19:49}, $x(\cdot)$ determines $f\in C$ such that $x(t)= x(s)+\int_s^tf(u,x(u))\, du$ for every $s,t\in J$. Moreover, up to a reordering of the functions $f_1,\dots, f_\nu$ whose $\ep/2$-neighborhoods provide a covering of $C$, assume that $p_j(f-f_1)\le \ep/2$. Then, for any $s,t\in J$ with $|t-s|<\delta$ we have
\begin{equation*}
\begin{split}
|x(t)-x(s)|&\le \int_s^t\big|f\big(u,x(u)\big)\big|\, du\\
&=\int_s^t\big|f\big(u,x(u)\big)-f_1\big(u,x(u)\big)\big|\, du+\int_s^tm_1(u)\, du\\
&\le p_j(f-f_1)+\int_s^tm_1(u)\, du\le \ep\, .
\end{split}
\end{equation*}
Hence, from the arbitrariness of $x(\cdot)\in C^I_j$, one has that  the set $C^I_j$ is equicontinuous.\par\vspace{0.05cm}
(iii) The proof of part (iii) is equal to the one of part (ii) of Theorem \ref{thm:ContODETthetaE} with the exception that  now $\Theta$ is no more determined by the $m$-bounds of the functions in $E$ but as in the statement. Notice that everything is consistent, since for any $I=[q_1,q_2]\subset\R$, with $q_1,q_2\in\Q$, and for any $j\in\N$, we have that $C^I_j\subset \K^I_j$, where the functions in $C^I_j$ are possibly extended by constants to the whole interval $I$, as before.
\end{proof}
Next, we state the result of continuity of the skew-product flow for the topology $\T_D$. Notice that, in analogy with Theorem~\ref{thm:contFlowTtheta}, we provide a result for both systems like \eqref{eq:solLCE} and like \eqref{eq:system} in the respective hulls. However, a major difference in the assumptions of the second case occurs, that is, $\mathrm{Hull}_{(\LC,\T_D)}(f)$ is required to be compact, due to the fact that \ref{thm:ContODETpE}(iii) is used to obtain the result. Incidentally, recall that a characterization of compactness of $\mathrm{Hull}_{(\LC,\T_D)}(f)$ is given in Corollary \ref{cor:compactnessHULL(f)}.
\par
As before, let us set some notation first. Considered $f\in\LC$, let us denote by $\U_1$ the subset of $\R\times\mathrm{Hull}_{(\LC,\T_D)}(f)\times\R^N$ given by
\begin{equation*}
\U_1=\bigcup_{\substack{g\in\mathrm{Hull}_{(\LC,\T_D)}(f)\\x_0\in\R^N}} \{(t,g,x_0)\mid t\in I_{g,x_0}\}\,,
\end{equation*}
and, if $\Theta=(\theta_j)_{j\in\N}$ is a suitable set of moduli of continuity, $F\in\TC\big(\R^{N\times N}\big)$ and $h\in \TC$, let us denote by $\U_2$ the subset of $\R\times\Hu\times\R^N\!\!\times\R^N$, with $\Hu=\mathrm{Hull}_{\left(\LC\times \TC\times\TC,\T_D\times\T_\Theta\times\T_\Theta\right)}(f,F,h)$, given by
\begin{equation*}
\U_2=\bigcup_{\substack{(g,G,k)\in \Hu\\ x_0\in\R^N}}\{(t,g,G,k,x_0,y_0)\mid t\in I_{g,x_0}\,,\; y_0\in\R\}\,.
\end{equation*}
\begin{thm}
Let $f\in\LC$ have $L^p_{loc}$-bounded $l$-bounds.\begin{itemize}
\item[(i)]
The set $\U_1$ is open in $\R\times\mathrm{Hull}_{(\LC,\T_D)}(f)\times\R^N$ and the map
\begin{equation*}
\begin{split}
\Phi_1\colon \ \U_1\subset \R\times\mathrm{Hull}_{(\LC,\T_D)}(f)\times\R^N\  &\to\ \mathrm{Hull}_{(\LC,\T_D)}(f)\times\R^N\, ,\\
\ (t,g,x_0)\qquad \qquad \qquad &\mapsto\qquad \big(g_t, x(t,g,x_0)\big)\,,\\
\end{split}
\end{equation*}
defines a local continuous skew-product flow on $\mathrm{Hull}_{(\LC,\T_D)}(f)\times\R^N$.
\item[(ii)] Furthermore, if $\mathrm{Hull}_{(\LC,\T_D)}(f)$ is compact and $\Theta$  is the suitable set of moduli of continuity given by {\rm Theorem \ref{thm:ContODETpE}(ii)}, and if $F\in\TC\big(\R^{N\times N}\big)$ and $h\in \TC$, then the  set $\U_2$ is open in $\R\times \Hu\times \R^N\!\!\times \R^N$, and the map
\begin{equation*}
\begin{split}
\qquad\quad\Phi_2\colon  \ \U_2\subset \R\times\Hu\times\R^N\!\!\times\R^N  &\to\ \qquad\qquad\quad \Hu\times\R^N\!\!\times\R^N\\
\ (t,g,G,k,x_0, y_0)\quad &\mapsto\ \big(g_t, G_t, k_t ,x(t,g,x_0),y(t,g,G,k,x_0, y_0)\big)\,,\\
\end{split}
\end{equation*}
defines a local continuous skew-product flow on $\Hu\times\R^N\!\!\times \R^N$.
\end{itemize}
\end{thm}
The proof is similar to the proof of Theorem~\ref{thm:contFlowTtheta}, and thus skipped.
\section{The linearized skew-product semiflow}\label{sec:linearizedSPflow}
Let $E\subset\LC$, and assume that every $f\in E$ is continuously differentiable with respect to $x$ for a.e. $t\in\R$ and that $J_xf\in \SC(\R\times\R^N,\R^{N\times N})$, where $J_xf$ is the Jacobian of $f$ with respect to the coordinates $x$. The classic theory of Carath\'eodory ODEs  provides the differentiability of the solutions with respect to the initial conditions when the respective vector fields are in $E$ (see Kurzweil \cite{book:Kurz}). In this section, under an additional hypothesis on $E$, granting the existence of a suitable set of moduli of continuity $\Theta$, we extend such conclusions to the solutions of Carath\'eodory differential equations whose vector fields are in a subset of $\mathrm{cls}_{(\LC,\T_\Theta)}(E)$ and may possibly not admit continuous partial derivative with respect to $x$. In particular, we introduce new types of continuous linearized skew-product semiflow in the spaces of Carath\'eodory functions.
\begin{thm}
Consider $E_1\subset\LC$ with $L^1_{loc}$-equicontinuous $m$-bounds, assume that all the functions in $E_1$ are continuously differentiable with respect to $x$ for a.e. $t\in\R$ and, for any $f\in E_1$, assume that $J_xf\in \SC(\R\times\R^N,\R^{N\times N})$, where $J_xf$ is the Jacobian of $f$ with respect to the coordinates $x$. Let $\Theta$ be like in {\rm Definition~\ref{def:THETAjseq}} and consider
\begin{equation*}
E=\mathrm{cls}_{(\LC\times\TC,\T_\Theta\times \T_\Theta})\{(f,  J_xf)\mid f\in E_1\}\,.
\end{equation*}
For any $(g,G)\in E$, if $x(t,g,x_0)$ and $y(t,g,G,x_0, y_0)$ are respectively the solutions of the Cauchy Problems
\begin{equation*}
\begin{cases}
\dot x=g(t,x)\\
x(0)=x_0
\end{cases}\quad \text{and}\qquad
\begin{cases}
\dot y=G\big(t,x(t,g,x_0)\big)\,y\\
y(0)= y_0
\end{cases}
\end{equation*}
defined for $t\in [T_0,T_1]\subset I_{g,x_0}$ (maximal interval of definition), then we have that
\begin{equation*}
\lim_{\ep\to0^+}\left| \frac{x(t,g,x_0+\ep y_0)-x(t,g,x_0)}{\ep}-y(t,g,G,x_0, y_0)\right|=0\, ,
\end{equation*}
uniformly for $t\in [T_0,T_1]$ and $y_0 \in B_1$.
\label{thm:linskewprod}
\end{thm}
\begin{proof}
For any $f\in E_1$, the result  is classic and can be found in \cite{book:Kurz} for example. For the sake of completeness, however, we include a proof of it. Let us fix $f\in E_1$ and simplify the notation denoting by $x(\cdot,x_0)=x(\cdot,f,x_0)$, and by $y(\cdot,x_0,y_0)=y(\cdot,f,J_xf,x_0, y_0)$; by the definition of solution, we have that
\begin{align*}
\frac{x(t,x_0+\ep y_0)-x(t,x_0)}{\ep}=y_0 +\frac{1}{\ep}\int_0^t\!\!\left[f\!\big(s,x(s,x_0+\ep y_0)\big)-f\!\big(s,x(s,x_0)\big)\right]ds\\
= y_0+\int_0^t\!\!\left(\int_0^1 \!\!J_xf\!\big(s,\xi_{\ep}(s,\alpha)\big)\, d\alpha\right)\frac{x(s,x_0+\ep y_0)-x(s,x_0)}{\ep}\, ds ,
\end{align*}
where $\xi_{\ep}(s,\alpha)=x(s,x_0)+\alpha\,[x(s,x_0+\ep y_0)-x(s,x_0)]$ is determined by the fundamental theorem of calculus. Furthermore, by definition
\begin{equation*}
y(t,x_0, y_0)= y_0+\int_0^tJ_xf\big(s,x(s,x_0)\big)\, y(s,x_0, y_0)\, ds\, .
\end{equation*}
Therefore, if $0\le t\leq T_1$ one has
\begin{align}
\nonumber &\left| \frac{x(t,x_0+\ep y_0)-x(t,x_0)}{\ep}-y(t,x_0, y_0)\right|\\
\nonumber &\;\le \int_0^t\bigg|\!\left(\int_0^1J_xf\big(s,\xi_\ep(s,\alpha)\big)\, d\alpha\right)\left[\frac{x(s,x_0+\ep y_0)-x(s,x_0)}{\ep}- y(s,x_0, y_0)\right]\!\bigg|\, ds\\
&\quad+\int_0^t\bigg|\left[\int_0^1\Big(J_xf\big(s,\xi_\ep(s,\alpha)\big)\,-J_xf\big(s,x(s,x_0)\big)\Big) d\alpha\right] y(s,x_0, y_0)\bigg|\, ds\,.
\label{eq:25.07-11:23}
\end{align}
Denote by $\eta_\ep(t,x_0,  y_0) $ the integral
\begin{equation*}
\int_0^t\bigg|\left[\int_0^1\Big(J_xf\big(s,\xi_\ep(s,\alpha)\big)\,-J_xf\big(s,x(s,x_0)\big)\Big) d\alpha\right] y(s,x_0, y_0)\bigg|\, ds\,,
\end{equation*}
and notice that,  if $\ep\to0$, then $\eta_\ep(t,x_0,  y_0)\to 0$ uniformly in $y_0\in B_1$ and $t\in[T_0,T_1]$ since $ y(\cdot,x_0, y_0)$ is bounded, for Theorem~\ref{thm:ContODETthetaE}, for $t\in[T_0, T_1]$ and $y_0\in B_1$, and since $J_xf\big(s,\xi_\ep(s,\alpha)\big)$ converges to $J_xf\big(s,x(s,x_0)\big)$ uniformly in $\alpha\in[0,1]$, $y_0\in B_1$ and $t\in[T_0,T_1]$ as $\ep\to0$. Moreover, recalling how $\xi_\ep(s,\alpha)$ is defined and once again thanks to Theorem~\ref{thm:ContODETthetaE}, we know that there exists $j\in \N$ such that $\|\xi_\ep(s,\alpha)\|_\infty<j$ for every $\ep\le1$, every $s\in[T_0,T_1]$ and every $\alpha\in [0,1]$. Thus, denoting by $m^j(\cdot)$  the optimal $m$-bound on $B_j$ for $J_xf$, we in particular have that $\int_0^1\|J_xf(s,\xi_\ep(s,\alpha))\|\, d\alpha\leq m^j(s)$. Then, from~\eqref{eq:25.07-11:23}, we deduce that
\begin{multline*}
\left| \frac{x(t,x_0+\ep y_0)-x(t,x_0)}{\ep}-y(t,x_0, y_0)\right| \\ \le \eta_\ep(t,x_0,y_0)+\int_0^t m^j(s) \left| \frac{x(s,x_0+\ep y_0)-x(s,x_0)}{\ep}-y(s,x_0, y_0)\right|\,ds\,,
\end{multline*}
and  applying Gronwall's inequality  we get
\begin{multline}\label{eq:24.10-11:40}
\left|\frac{x(t,x_0+\ep y_0)-x(t,x_0)}{\ep}-y(t,x_0, y_0)\right|\\
\begin{split}
&\le  \eta_\ep(t,x_0,  y_0)+ \int_0^t\eta_\ep(s,x_0,  y_0)\, m^j(s) \, \exp\left(\int_s^t \!\!\!m^j(r)\, dr\right)ds\\
&\le \eta_\ep(t,x_0, y_0)+ c\,\int_0^{T_1}\eta_\ep(s,x_0, y_0)\, m^j(s)ds\,,
\end{split}
\end{multline}
where the positive constant $c$ satisfies $c\ge\exp\left(\int_0^{T_1}\!\!\!m^j(r)\, dr\right)ds$. Notice that, as $ \ep\to 0 $, the right-hand side of \eqref{eq:24.10-11:40} vanishes  uniformly for $t\in[0,T_1]$ and $y_0 \in B_1$. A similar inequality for $T_0\le t\le 0$ yields to
\begin{equation}\label{eq:linCLASSIC}
\lim_{\ep\to0^+}\frac{x(t,f,x_0+\ep y_0)-x(t,f,x_0)}{\ep}=y(t,f,J_xf,x_0, y_0)\,,
\end{equation}
uniformly in $t\in[T_0,T_1]$ and $y_0 \in B_1$, which proves the result for $f\in E_1$.\par\vspace{0.05cm}
Now consider $(g,G)\in E$ and let $(f_n)_\nin$ be a sequence of functions in $E_1$ such that $(f_n,J_xf_n)_\nin$ converges to $(g,G)$ in $\T_\Theta$. From \eqref{eq:linCLASSIC}, we have that  for any $\nin$ and for any $x_1,x_2\in\R^N$, with $|x_1-x_2|\le 1$, the following equality holds
\begin{equation}
x(t,f_{n},x_1)-x(t,f_{n},x_2)=\int_0^1y(t,f_n,J_xf_n, \alpha\,x_1+(1-\alpha)\,x_2, x_1-x_2)\,  d\alpha\, .
\label{eq:27.07-12:31}
\end{equation}
Moreover, let $C$ be the set $\{(f_n,J_xf_n)\mid\nin\}\cup(g,G)$, and let $B$ be a closed ball in $\R^N\!\times\R^N$ containing $\{(x,y)\in\R^N\!\!\times\R^N\mid |x-x_0|\le1,y\in B_1\}$. Then, if $[T_0,T_1]\subset I_{g,x_0}$, thanks to Theorem~\ref{thm:ContODETthetaE}, we have that the application from $[T_0,T_1]\times C\times B$ into $\R^N\!\!\times\R^N$ defined by
\begin{equation*}
(t,h,H, x_0, y_0)\mapsto \big(x(t,h,x_0), y(t,h,H, x_0, y_0)\big)\, ,
\end{equation*}
is uniformly continuous and bounded on $[T_0,T_1]\times C\times B$. Thus, we have that, as $\nti$, equation \eqref{eq:27.07-12:31} becomes
\begin{equation}
x(t,g,x_1)-x(t,g,x_2)=\int_0^1y(t,g,G, \alpha\,x_1+(1-\alpha)\,x_2,x_1-x_2)\,  d\alpha\, .
\label{eq:27.07-12:40}
\end{equation}
Eventually, if in \eqref{eq:27.07-12:40} we consider $x_2=x_0$ and $x_1=x_0+\ep  y_0$, where $\ep\le1$ and $y_0\in B_1$. Then, one has
\begin{multline*}
\left|\frac{x(t,g,x_0+\ep y_0)-x(t,g,x_0)}{\ep}-y(t,g ,G,x_0, y_0)\right|\\
=\left|\frac{1}{\ep}\int_0^1y(t,g,G, x_0+\alpha\ep y_0,\ep y_0)\,  d\alpha-y(t,g ,G,x_0, y_0)\right|\\
\le\int_0^1|y(t,g,G, x_0+\alpha\ep y_0, y_0)-y(t,g ,G,x_0, y_0)|\,  d\alpha \, .
\end{multline*}
Then, applying Theorem~\ref{thm:ContODETthetaE} once again, when $\ep\to 0$, and reasoning as before, we obtain the thesis.
\end{proof}
\begin{defn}
Let $f\in \LC$ be continuously differentiable with respect to $x$ for a.e. $t\in\R$ and with $L^1_{loc}$-equicontinuous $m$-bounds. Let $\Theta$ be defined as  in Remark~\ref{rmk:THETAjfunc} and denote by $J_xf\in \SC\big(\R^{N\times N}\big)$ the Jacobian of $f$ with respect to the coordinates $x$ and let us denote by $\Hu=\mathrm{Hull}_{(\LC\times\TC,\T_\Theta\times\T_\Theta)}(f,  J_xf)$. If $\U$ is the subset of $\R\times\Hu\times\R^N\!\!\times\R^N$ given by
\begin{equation*}
\U=\bigcup_{\substack{(g,G)\in\Hu\\ x_0\in\R^N}}\left\{ (t,g,G,x_0,y_0)\mid t\in I_{g,x_0},\,y_0\in\R^N\right \}\,,
\end{equation*}
then, we call  \emph{a linearized skew-product semiflow} the map
\begin{equation*}
\begin{split}
\Psi\colon \U\subset \R\times\Hu\times\R^N\!\!\times\R^N\quad&\to\quad\ \qquad\qquad\Hu\times\R^N\!\!\times\R^N\\
(t,g,G, x_0, y_0)\qquad\quad&\mapsto\quad \big(g_t,G_t, x(t,g,x_0), y(t,g,G, x_0, y_0)\big)\, ,
\end{split}
\end{equation*}
\end{defn}
The use of the name ``linearized skew-product semiflow'' is meaningful  because, according to Theorem~\ref{thm:linskewprod} we have that $\partial x(t,g,x_0)/\partial x_0\cdot  y_0= y(t,g,G,x_0, y_0)$ for every $(g,G)\in\Hu$ and every $t\in I_{g,x_0}$, and therefore  in particular when $G\in\TC\big(\R^{N\times N}\big)\setminus\SC\big(\R^{N\times N}\big)$, i.e. when $g$ does not have continuous partial derivatives with respect to $x$ for almost every $t\in\R$.
\par\vspace{0.05cm}
Next we give a simple example, when $N=1$, exhibiting such a phenomenon.
\begin{exmpl}
Consider the continuous function $H\colon \R\to\R$ such that $H(t)=0$ if $t<0$ and, for any $\nin$, $H(t)$ is defined in the interval $[4n,4n+4]$ as follows:
\begin{equation*}
H(t)=
\begin{cases}
(1+n)(t-4n), & \text{if $t\in I^1_n=\left[4n,4n+\frac {1}{n+1}\right]$,} \\[.2cm]
1, & \text{if $t\in I^2_n=\left[4n+\frac {1}{n+1},4n+2-\frac {1}{n+1}\right]$,} \\[.2cm]
-(1+n)(t-4n-2), & \text{if $t\in I^3_n=\left[4n+2-\frac {1}{n+1}, 4n+2+\frac {1}{n+1}\right]$,} \\[.2cm]
-1, & \text{if $t\in I^4_n=\left[4n+2+\frac {1}{n+1}, 4n+4-\frac {1}{n+1}\right]$,} \\[.2cm]
(1+n)(t-4n-4), & \text{if $t\in I^5_n=\left[ 4n+4-\frac {1}{n+1},4n+4\right]$} \,.\\[.2cm]
\end{cases}
\end{equation*}
Notice that as $\nti$ the measures of $I^1_n$, $I^3_n$ and $I^5_n$ go to zero, whereas the measures of $I^2_n$ and $I^4_n$ go to $2$. Thus, if we consider the sequence of translations of $H$ given by $\big(H_{4k}(\cdot)\big)_{k\in\N}$, we have that
\begin{equation*}
H_{4k}(t)\xrightarrow{k\to\infty} \overline H(t)=
\begin{cases}
0 , & \text{if $t=4n$,} \\
1, & \text{if $t\in (4n,4n+2)$,} \\
0 , & \text{if $t=4n+2$,} \\
-1, & \text{if $t\in (4n+2, 4n+4)$,}\\
\end{cases}\qquad n\in\Z\,,\ \forall \, t\in\R\,.
\end{equation*}
\begin{figure}[]
\centering
\includegraphics[width=\textwidth]{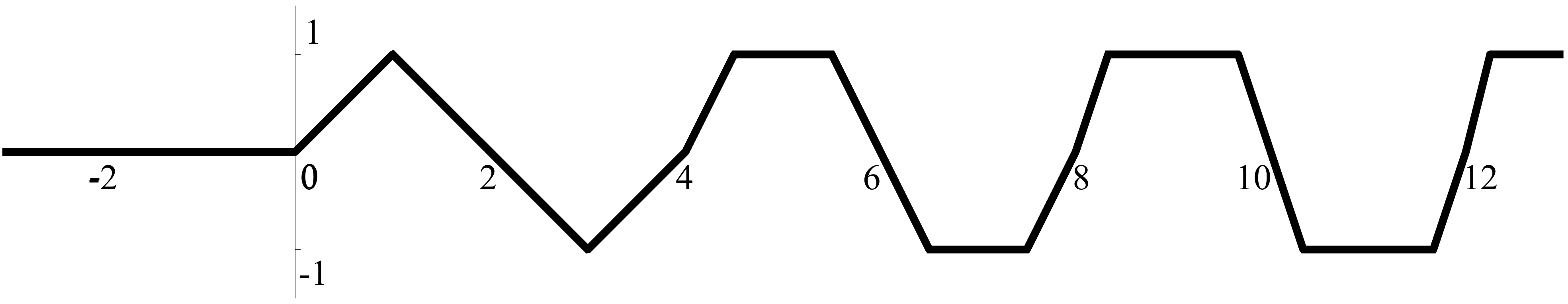}
\caption{The function $H(t)$.}
\end{figure}
\begin{figure}[]
\centering
\includegraphics[width=\textwidth]{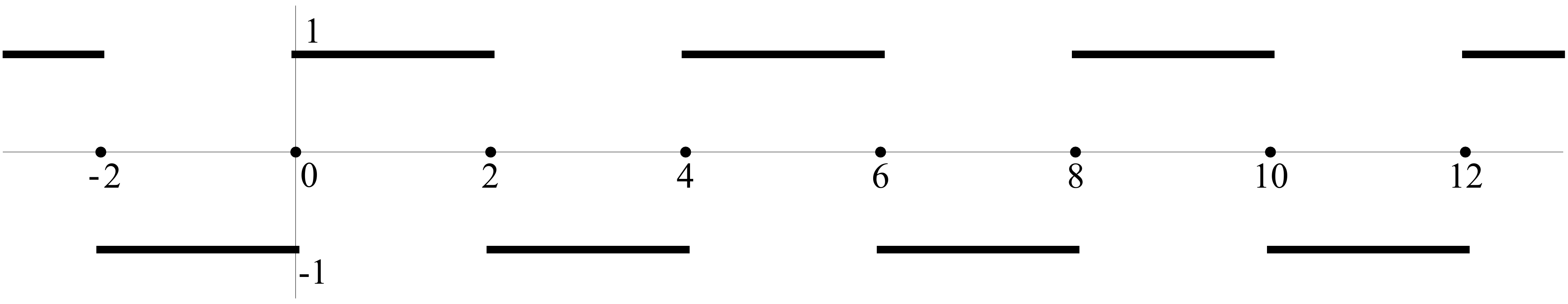}
\caption{The function $\overline H(t)$.}
\end{figure}
Now consider the function $h\colon\R\to\R$ defined by
\begin{equation*}
h(t)=\int_0^t H(s)\,ds\,.
\end{equation*}
Notice that $h\in C^1(\R)$, $|h(t)|\le 2$ and $|h'(t)|\le1$ for any $t\in\R$; consequently, $h$ has Lipschitz constant equal to $1$. Therefore, the $\cls\{h_\tau(\cdot)\mid\tau\in\R\}$ is compact in $C(\R)$ with the usual norm and $(h_{4k}(\cdot))_{k\in\N}$ converges uniformly on compact sets, to some bounded and Lipschitz function $\overline h$. It is easy to check that $\overline h(t)=\int_0^t\overline H(s)\, ds$. Then, consider the functions $f\in\LC$ and $F\in\SC$ defined by
\begin{equation*}
f(t,x)=h\left(t+\frac{x}{3}\right)\qquad\text{and}\qquad F(t,x)=\frac{1}{3}\, H\left(t+\frac{x}{3}\right)\,.
\end{equation*}
Moreover, taking into account the unique modulus of continuity $\theta(t)=2t$ calculated like in Remark \ref{rmk:THETAjfunc}, let us consider the $\mathrm{Hull}_{(\LC\times\TC,\T_\Theta\times\T_\Theta)}(f,  F)$ where, according to the notation of Section \ref{secContFlowODEs}, we may write $F=J_xf$.
\par
Now let us consider the following family of differential systems whose vector fields are in the $\mathrm{Hull}_{(\LC\times\TC,\T_\Theta\times\T_\Theta)}(f,  F)$,
\begin{equation}
\begin{cases}
\dot x=f_{4k}(t,x)\\
\dot y= F_{4k}\big(t,x(t)\big)\, y
\end{cases}\ ,\quad k\in\N\,.
\label{eq:28.07-13:27}
\end{equation}
One can easily check that, for any $k\in\N$, the second differential equation in \eqref{eq:28.07-13:27} is the variational equation of the first one, evaluated along the solution $x(t)$ of the first equation. Moreover, since $h_{4k}(t)\to\overline h(t)$ uniformly on compact sets as $k\to\infty$, then $f_{4k}(t,x)\to g(t,x)$ in $\T_D$ as $k\to\infty$, where $g(t,x)=\overline h\left(t+x/3\right)$. Actually, since $f$ satisfies the hypothesis of Corollary \ref{cor:HullT1=HullT2}, the convergence $f_{4k}(t,x)\to g(t,x)$ holds for any of the considered topologies. Furthermore, we claim  that $F_{4k}\to G$ in $\T_\theta$, where $G(t,x)=(1/3)\, \overline H\left(t+x/3\right)$. Indeed, for any compact interval $I\subset\R$, for any $j\in\N$, and for any $z(\cdot)\in C(I,B_j)$ with $\theta(t)=2t$ as modulus of continuity,  we have
\begin{equation*}\begin{split}
\int_I\big|F_{4k}\big(&t,z(t)\big)-G\big(t,z(t)\big)\big|^pdt=\frac{1}{3^p}\int_I\left|H_{4k}\left(t+\frac{z(t)}{3}\right)-\overline H\left(t+\frac{z(t)}{3}\right)\right|^pdt\\
&\le3^{1-p}\int_I\left|H_{4k}\left(t+\frac{z(t)}{3}\right)-\overline H\left(t+\frac{z(t)}{3}\right)\right|^p\left|1+\frac{z'(t)}{3}\right|dt\\
&\le\frac{1}{3^p}\int_{I+[-j/3,j/3]}\left|H_{4k}\left(s\right)-\overline H\left(s\right)\right|^pds\,,
\end{split}
\end{equation*}
where, in the first inequality, $z'(t)$ is the derivative almost everywhere of $z(t)$, whose existence is granted by the fact that $z(t)$ is Lipschitz, and we use the fact that  $1/3\le|1+z'(t)/3|$ for every $t\in I$. Moreover, the theorem of change of variables for the measurable case (see Hewitt and Stromberg  \cite[Theorem 20.5]{book:HS}) has been used in the last inequality. Therefore, since $(H_{4k}(s))_{k\in\N}$ converges almost everywhere to $\overline  H(s)$, the Lebesgue theorem of dominated convergence gives us the result. Finally, notice that $G(t,x)$ is a Borel function. Hence, thanks to Theorem \ref{thm:meas+conv=>ThetaC}, we have that $G\in\TC$, and it is straightforward to see that $G\notin\SC$.
\end{exmpl}


\begin{thebibliography}{99}
\bibitem{book:LA} \textsc{L. Arnold}: \emph{Random Dynamical Systems}, \textrm{Springer-Verlag, Berlin, Heidelberg, 1998.}
\bibitem{paper:ZA1} \textsc{Z. Artstein}: Topological dynamics of an ordinary differential equation, \emph{J. Differential Equations} \textbf{23} (1977), 216--223.
\bibitem{paper:ZA2} \textsc{Z. Artstein}: Topological dynamics of ordinary differential equations and Kurzweil equations, \emph{J. Differential Equations} \textbf{23} (1977), 224--243.
\bibitem{paper:ZA3} \textsc{Z. Artstein}: The limiting equations of nonautonomous ordinary differential equations, \emph{J. Differential Equations} \textbf{25} (1977), 184--202.
\bibitem{paper:AW} \textsc{B. Aulbach, T. Wanner}: Integral manifolds for Carath\'eodory type differential equations in Banach spaces, \emph{Six Lectures on Dynamical Systems}
(B. Aulbach \& F. Colonius eds), World Scientific, Singapore, 1996, 45--119.
\bibitem{paper:SBRS} \textsc{S.I. Bodin, R.J. Sacker}: A new approach to asymptotic diagonalization of linear differential systems, \emph{J. Dynam. Differential Equations} \textbf{12} (2000), 229--245.
\bibitem{book:CLR} \textsc{A. Carvalho, J.A. Langa, J. Robinson}: \emph{Attractors for infinite-dimensional non-au\-tono\-mous dynamical systems}, \textrm{Springer-Verlag New York, 2013.}
\bibitem{book:CR} \textsc{R.E. Castillo, H. Rafeiro}: \emph{An Introductory Course In Lebesgue Spaces}, \textrm{Springer, Switzerland, 2016.}
\bibitem{paper:CHLE} \textsc{S.-N. Chow, H. Leiva}: Dynamical spectrum for time dependent linear systems in Banach spaces, \emph{Japan J. Indust. Appl. Math.} \textbf{11} (1994), 379--415.
\bibitem{book:CL} \textsc{E.A. Coddington, N. Levinson}: \emph{Theory Of Ordinary Differential Equations}, \textrm{McGraw-Hill, New York, 1955.}
\bibitem{paper:AJH} \textsc{A.J. Heunis}: Continuous dependence of the solutions of an ordinary differential equation, \emph{J. Differential Equations} \textbf{54} (1984), 121--138.
\bibitem{book:HS} \textsc{E. Hewitt, K. Stromberg}: \emph{Real and abstract analysis. A modern treatment of the theory of functions of a real variable.} Second printing corrected, \textrm{Springer-Verlag, New York, 1969.}
\bibitem{book:JONNF} \textsc{R. Johnson, R. Obaya, S. Novo, C. Nu\~nez, R. Fabbri}: \emph{Nonautonomous Linear Hamiltonian Systems: Oscillation, Spectral Theory And Control.}, Developments in Mathematics \textbf{36}, \textrm{Springer, Switzerland, 2016.}
\bibitem{book:Kall} \textsc{O. Kallenberg}: \emph{Random Measures.} Third Edition, \textrm{Akademie-Verlag, Berlin, 1983.}
\bibitem{book:Kurz} \textsc{J. Kurzweil}: \emph{Ordinary Differential Equations},  Studies in Applied Mechanics \textbf{13}, \textrm{Elsevier, Amsterdam, 1986.}
\bibitem{book:RMGS} \textsc{R.K. Miller, G. Sell}: Volterra Integral Equations and Topological Dynamics, \emph{Mem. Amer. Math. Soc.}, no. 102, \textrm{Amer. Math. Soc., Providence, 1970.}
\bibitem{paper:RMGS1} \textsc{R.K. Miller, G. Sell}: Existence, uniqueness and continuity of solutions of integral equations, \emph{Ann. Math. Pura Appl.} \textbf{80} (1968), 135--152;  Addendum: ibid. \textbf{87} (1970), 281--286.
\bibitem{paper:LWN} \textsc{L.W. Neustadt}: On the solutions of certain integral-like operator equations. Existence, uniqueness and dependence theorems, \emph{Arch. Rational Mech. Anal.} \textbf{38} (1970), 131--160.
\bibitem{paper:CPMR} \textsc{C. P\"otzsche, M. Rasmussen}: Computation of integral manifolds for Carath\'{e}odory differential equations, \emph{J. Numer. Anal.} \textbf{30} (2010), 401--430.
\bibitem{paper:GS1} \textsc{G. Sell}: Compact sets of nonlinear operators, \emph{Funkcial. Ekvac.} \textbf{11} (1968), 131--138.
\bibitem{book:GS} \textsc{G. Sell}: \emph{Topological Dynamics and Ordinary Differential Equations}, \textrm{Van Nostrand-Reinhold, London, 1971.}
\bibitem{paper:WSYY} \textsc{W. Shen, Y. Yi}: Almost automorphic and almost periodic dynamics in skew-product semiflows, \emph{Mem. Amer. Math. Soc.} \textbf{136}, no. 647, \textrm{Amer. Math. Soc., Providence,  1998}.
\bibitem{paper:SS} \textsc{S. Siegmund}: Dichotomy spectrum for nonautonomous differential equations, \emph{J. Dynam. Differential Equations} \textbf{14} (2002), 243--258.
\end{thebibliography}
\end{document}